\documentclass{article}
\usepackage[utf8]{inputenc}

\usepackage{float}
\usepackage{graphicx}
\usepackage{amsmath}
\usepackage{enumitem}
\usepackage{amscd}
\usepackage{amsthm}
\usepackage{mathtools}
\usepackage{amssymb}
\usepackage{tikz-cd}
\usepackage{caption}
\usepackage{dsfont}
\usepackage[font=small,labelfont=bf,width=1\textwidth]{caption}
\usepackage{multicol}
\usepackage{subfig}
\usepackage[left=3.5cm,top=4cm,right=3.5cm,bottom=4cm]{geometry}

\usepackage{siunitx}
\usepackage{lipsum}
\usepackage{xfrac}
\usepackage{marvosym}
\usepackage{epic}
\usepackage{wasysym}
\usepackage{epstopdf}
 
\usepackage{hyperref}
\usepackage{wrapfig}
\usepackage{enumitem}
\usepackage{enumerate}
\usepackage{color}
\usepackage[UKenglish]{datetime}

\theoremstyle{plain}
\newtheorem{theorem}{Theorem}[section]
\newtheorem{lemma}[theorem]{Lemma}
\newtheorem{corollary}[theorem]{Corollary}
\newtheorem{proposition}[theorem]{Proposition}

\newtheorem{remark}[theorem]{Remark}
\newtheorem{conjecture}[theorem]{Conjecture}

\theoremstyle{definition}
\newtheorem{definition}[theorem]{Definition}
\newtheorem{example}[theorem]{Example}

\mathtoolsset{showonlyrefs}

\title{Jacobi polynomials on the simplex of Hermitian matrices and zonal spherical functions on partial flag manifolds}
\author{Teije Kuijper\footnote{t.kuijper@math.au.dk}}
\date{\today}

\begin{document}

\maketitle

\begin{abstract}
    We introduce the radial part of the Laplace--Beltrami operator on partial flag manifolds and study its eigenfunctions, the elementary zonal spherical functions.
    These eigenfunctions can be related to orthogonal polynomials on the simplex of Hermitian matrices invariant under simultaneous conjugation by unitary matrices with respect to the complex matrix-variate Dirichlet distribution.
    Using a matrix-variate version of Koornwinder's method for constructing orthogonal polynomials in multiple variables from orthogonal polynomials in one variable, we construct a family of pairwise orthogonal polynomials in terms of Hermitian Jacobi polynomials and conjecture that some of these are elementary zonal spherical functions.
    Furthermore, we recall the definition of multivariate Schur polynomials and show how they can be used to obtain information about elementary zonal spherical functions.
    In particular, we give some partial results in the direction that the radial part of the Laplace--Beltrami operator on partial flag manifolds is strongly triangular with respect to the multivariate Schur polynomials.
    The elementary zonal spherical functions can then be constructed by applying the Gram--Schmidt orthogonalisation process to the multivariate Schur functions.
\end{abstract}

\tableofcontents

\section*{Introduction}
The theory of zonal spherical functions on symmetric spaces is rich and has proven to be fruitful.
Despite this fact, similar theories for classes of non-symmetric homogeneous spaces seem few and far between.
The main aim of this paper is to suggest and study a similar notion of \emph{zonal spherical functions} for partial flag manifolds.
Since the groups acting on the left and the right are different, they might more accurately be called \emph{intertwining functions}.

Flag manifolds are a special family of homogeneous spaces that arise in a myriad of different contexts, among others, representation theory \cite{Wolf1998}, stochastic analysis \cite{Baudoin2024-is, kuijper2025}, complex geometry \cite{Borel1954, CORREA2019109}, Einstein manifolds \cite[Chapter 8.]{Besse2007-en} and Poisson geometry \cite[Example 4.22]{Crainic2021-tz}.
Generalised flag manifolds have been classified \cite[\textsection 7.4.]{Arvanitoyeorgos2003-tp}, in which the partial flag manifolds form the first infinite family.
They are of $A_n$-type in the Dynkin classification scheme.
While also being of huge importance in themselves, the partial flag manifolds are often used as a testing ground for a more abstract theory which works for all generalised flag manifolds simultaneously.
This is also a strong motivation behind the current paper.

Radial coordinates on flag manifolds are essential in the study of Brownian motions on flag manifolds \cite{baudoin2025fullflag, kuijper2025}.
By general theory, the heat kernel of the radial part of the Brownian motion can be expanded in terms of eigenfunctions of the radial part of the Laplace--Beltrami operator \cite[Theorem 4.52.]{Baudoin2014-ca}.
The heat kernel of the radial part of a Brownian motion on flag manifolds and related processes also appear in the expression for the joint characteristic function of the area processes on the partial flag manifolds with blocks of equal size \cite[Theorem 4.8.]{kuijper2025}.
The original motivation behind this paper was to understand these heat kernels better.

As a homogeneous space, the partial flag manifold $F_{d_1,\dots ,d_k}(\mathbb{C}^n)$ of signature $(d_1,\dots ,d_k)$ is given by
\begin{equation*}
    F_{d_1,\dots ,d_k}(\mathbb{C}^n) =\mathbf{U}(n)/(\mathbf{U}(n_1)\times\dots\times\mathbf{U}(n_{k+1})) ,
\end{equation*}
where $n_j:=d_j-d_{j-1}$, $d_0:=0$ and $d_{k+1}:=n$.
The radial part of $F_{d_1,\dots ,d_k}(\mathbb{C}^n)$ is defined to be the double quotient space
\begin{equation*}
    \mathbf{U}(n-m)\times\mathbf{U}(m)\backslash\mathbf{U}(n)/\mathbf{U}(n_1)\times\dots\times\mathbf{U}(n_{k+1}) .
\end{equation*}
for some $m\leq\min (n_1,\dots ,n_{k+1})$.
Coordinates on the open dense subset
\begin{equation*}
    \mathcal{D}_m:=\left\{\begin{pmatrix}
        W_1 & \dots & W_{k+1} \\ Z_1 &\dots & Z_{k+1}
    \end{pmatrix}\in\mathbf{U}(n)\big| Z_{j}\in\mathbb{C}^{m\times n_j}, \det (Z_jZ_j^*)\neq 0, 1\leq j\leq k+1\right\}
\end{equation*}
of the double quotient space are given by
\begin{equation*}
    (Z_1Z_1^*,\dots ,Z_{k}Z_k^*)
\end{equation*}
up to simultaneous conjugation with $m\times m$ unitary matrices.
These coordinates reside in the simplex of Hermitian matrices
\begin{equation*}
    \Sigma^m_k :=\left\{ (\Lambda_1,\dots ,\Lambda_{k})\in\mathrm{Her}(m)^k\biggm| \sum_{j=1}^{k}\Lambda_j\leq I_m\textrm{ and }\Lambda_j\geq 0 \,\textrm{ for }1\leq j\leq k\right\}
\end{equation*}
(quotiented by the diagonal action of $\mathbf{U}(m)$).

The radial part of the Laplace--Beltrami operator $\mathcal{G}^m_{(n_1 -m/2,\dots ,n_{k+1} -m/2)}$ on $F_{d_1,\dots ,d_k}(\mathbb{C}^n)$ can be determined using stochastic calculus.
We will consider the operator $\mathcal{G}^m_{(n_1-m/2,\dots ,n_{k+1}-m/2)}$ more generally for real $\kappa_j :=n_j-m/2>m/2-1$, which we will call the Hermitian Jacobi operators on the simplex.
It is symmetric with respect to the complex matrix-variate Dirichlet distribution.
To get some idea of the eigenfunctions of $\mathcal{G}^m_{(\kappa_1,\dots ,\kappa_{k+1})}$, we study orthogonal polynomials with respect to the complex matrix-variate Dirichlet distributions invariant under simultaneous conjugation by unitary matrices.
We are able to construct a family of pairwise invariant orthogonal polynomials \eqref{eq:matrix-Jacobi-polynomials-simplex}, the product Hermitian Jacobi polynomials on the simplex, using a matrix-variate version of Koornwinder's method, proposition \ref{prop:naive-Koornwinders-method}, for constructing orthogonal polynomials in multiple variables from orthogonal polynomials in one variable, some of which we conjecture to be eigenfunctions of $\mathcal{G}^m_{(\kappa_1,\dots ,\kappa_{k+1})}$.
Considerable effort is put into showing that the corresponding expressions are in fact polynomials invariant under simultaneous conjugation by unitary matrices.
They reduce to the Jacobi polynomials on the simplex when $m=1$, and to the Hermitian Jacobi polynomials when $k=1$, which are known to be eigenfunctions of $\mathcal{G}^1_{(\kappa_1,\dots ,\kappa_{k+1})}$ and $\mathcal{G}^m_{(\kappa_1,\kappa_2)}$ respectively.

In general, the product Hermitian Jacobi polynomials on the simplex will not form a basis of the space of invariant polynomials.
To get an idea of possible other eigenfunctions we suggest a theory for partial flag manifolds in the style of the hypergeometric function theory of Heckman \&\ Opdam \cite{HeckmanOpdam1987, HECKMAN1995}.
The main objects of relevance are the multivariate Schur polynomials from the statistics literature \cite{JoseA2011}, where they are called complex invariant polynomials.
These polynomials are recalled in definition \ref{def:multivariate-Schur-polynomials} and form a basis of the space of invariant polynomials.
The multivariate Schur polynomials are indexed by $k+1$ partitions, but are not uniquely defined in general.
We show that the Jacobi operator on the simplex is strongly triangular with respect to the multivariate Schur polynomials and give some partial results in this direction.
The Hermitian Jacobi polynomials on the simplex are then defined in definition \ref{def:Hermitian-Jacobi-polynomials-simplex} by applying the Gram--Schmidt process to multivariate Schur polynomials with respect to a lexicographical extension of the natural partial order on partitions.
We suggest a possible form \eqref{eq:general-Hermitian-Jacobi-polynomials-simplex} of these polynomials when $k=2$, curiously (naive) extensions of this expression to $k>2$ do not give polynomials.
Effort is taken throughout the text to show that the theory and the conjectures reduce to known results when $k=1$ or $m=1$.

Some other novel results include a version of the Jacobi--Trudy identity for Hermitian Jacobi polynomials, an explicit expression of these polynomials in terms of the coefficients of a Hermitian matrix, and a notion of degree for polynomials of Hermitian matrix arguments invariant under simultaneous conjugation.

The paper is organised as follows: section \ref{chap:symmetric-polynomials} recalls the theory of symmetric polynomials and the Hermitian Jacobi polynomials formulated in terms of polynomials of a Hermitian matrix argument invariant under conjugation by unitary matrices; section \ref{chap:invariant-polynomials} extends this theory to polynomials of multiple Hermitian matrix arguments, in particular a matrix-variate version of Koornwinder's method is given, and multivariate Hermitian Jacobi polynomials are introduced; and lastly section \ref{chap:Jacobi-polynomials} introduced radial coordinates on partial flag manifolds and the Hermitian Jacobi polynomials on the simplex and shows how they are related to each other, furthermore an explicit family of pairwise orthogonal polynomials with respect to the complex matrix-variate Dirichlet measure is given. 

I would like to thank Nizar Demni, Jan Frahm, Gert Heckman and Max van Horssen for the many invaluable discussions.

\section{Symmetric polynomials}\label{chap:symmetric-polynomials}
In this section we discuss the basics of the theory of symmetric functions.
Special care is taken to explain the equivalence between symmetric polynomials and polynomials of a single Hermitian matrix argument invariant under conjugation by unitary matrices, which is well-known in the field of symmetric functions.
We provide an overview for the bases of the symmetric polynomials relevant to the rest of this paper.
In particular, we emphasise an equivalence between the symmetric monomial polynomials and the Schur polynomials.
Furthermore, we provide a brief overview of Hermitian Jacobi polynomials, which are certain Heckman--Opdam polynomials of type $BC$.

The section on Hermitian Jacobi polynomials, contains a version of the Jacobi--Trudy identity and an expression of these polynomials in terms of the matrix coefficients.
Even though these expression follow relatively easy from the proof of similar expressions for the Schur functions, they seem to be new.

For the rest of this section we will fix an integer $m\geq 1$.
A polynomial $P:\mathbb{R}^m\rightarrow\mathbb{R}$ is called symmetric if
\begin{equation*}
    P(\lambda_{\sigma (1)},\dots ,\lambda_{\sigma (m)}) =P(\lambda_1 ,\dots ,\lambda_m)
\end{equation*}
for all $\sigma\in S_m$, where $S_m$ is the group of permutations of $\{ 1,\dots ,m\}$.
The space of symmetric polynomials will be denoted by $P(\mathbb{R}^m)^{S_m}$ and the space of homogeneous symmetric polynomials of total degree $r$ by $P_r(\mathbb{R}^m)^{S_m}$.

\subsection{Integer partitions}
Central to the theory of symmetric polynomials is the concept of a partition.

\begin{definition}
A \emph{partition} $\tau$ (of \emph{length} $m$) is a multi-index $\tau =(\tau_1 ,\dots ,\tau_k)\in\mathbb{N}^m$ such that $\tau_1\geq\dots\geq\tau_m\geq 0$.
The set of all partitions (of length $m$) will be denoted by $\mathcal{P}$.
The \emph{weight} of a parition is the number $|\tau| :=|\tau_1|+\dots +|\tau_k|$.
We will write $\tau\vdash r$ if $\tau$ has weight $r$.
\end{definition}

The set of partitions $\mathcal{P}$ can be totally ordered by the \emph{graded reverse lexicographic ordering}, which is given by
\begin{equation*}
    \tau <_{\mathrm{rlg}}\rho\Leftrightarrow |\tau| <|\rho|\textrm{ or }\tau_{\ell}=\rho_{\ell}\textrm{ for all }1\leq\ell <j\textrm{ and }\tau_{j}<\rho_{j}\textrm{ for some } 1\leq j\leq m
\end{equation*}
for two partitions $\tau$ and $\rho$.
It turns out however that a total order is too rigid for many applications.
The partitions can also be partially ordered by the \emph{graded natural} or \emph{graded dominating ordering} which is given by
\begin{equation*}
    \tau\leq\rho\Leftrightarrow |\tau|< |\rho|\textrm{ or }\sum_{j=1}^{\ell} \tau_j\leq \sum_{j=1}^{\ell}\rho_j\textrm{ for all } 1\leq\ell\leq m,
\end{equation*}
for two partition $\tau$ and $\rho$.

We have the implication $\tau\leq\rho\Rightarrow \tau\leq_{\mathrm{rlg}}\rho$.
In particular, the graded reverse lexicographical ordering is a totally ordered extension of the graded natural ordering.
From now on we will always assume that the partitions are partially ordered by the natural ordering unless stated otherwise.

We will write $\rho\subseteq\tau$ if $\rho_j\leq\tau_j$ for all $1\leq j\leq m$.
We will also occasionally use the following terminology.

\begin{definition}
    A \emph{semistandard} or \emph{column strict tableau} is a sequence of partitions
    \begin{equation*}
        \rho :=\lambda^{(0)}\subseteq\dots\subseteq\lambda^{(r)} =:\tau
    \end{equation*}
    such that $|\tilde{\lambda}^{(j)}|-|\tilde{\lambda}^{(j-1)}|\leq 1$ for all $1\leq j\leq r$.
    Here $\tilde{\lambda}^{(j)}_{\ell}:=\#\{ i\mid\lambda_{\ell}^{(j)}\geq i\}$ is the \emph{conjugate partition} of $\lambda^{(j)}$.
    The partition $\rho-\mu$ is called the \emph{shape} of the tableaux and the partition $(|\lambda^{(1)}|-|\lambda^{(0)}|,\dots ,|\lambda^{(r)}|-|\lambda^{(r-1)}|)$ the \emph{weight}.
\end{definition}

\subsection{Bases, generators and equivalence between symmetric polynomials and invariant polynomials of a Hermitian matrix argument}
In this subsection we will describe a few standard bases for the vector space of symmetric polynomials relevant to the rest of this paper.
Note that all the basis vectors are homogeneous polynomials.
By symmetry the \emph{monomial symmetric polynomials}, i.e. polynomials of the form
\begin{equation*}
    m_{\tau}(\lambda_1,\dots ,\lambda_k) :=\sum_{\sigma\in S_k}\lambda_{\sigma (1)}^{\tau_1}\cdots\lambda_{\sigma (k)}^{\tau_k}
\end{equation*}
for some partition $\tau$, form a basis of the space of symmetric polynomials.
The partitions can thus be used to index bases of the space of symmetric polynomials.
The largest partition with respect to the reverse lexicographic ordering appearing in the expansion of a symmetric polynomial $P$ is called the \emph{absolute degree} of $P$.
The weight of the absolute degree is called the \emph{total degree}.
Note that this agrees with the usual notion of total degree for multivariate polynomials.

The \emph{elementary symmetric polynomials} are defined by
\begin{equation*}
    e_{r}(\lambda_1,\dots ,\lambda_m ):=\sum_{1\leq \ell_1<\dots <\ell_{r}\leq m}\lambda_{\ell_1}\cdots\lambda_{\ell_{r}}
\end{equation*}
for $r\in\mathbb{N}$.
By the fundamental theorem of symmetric polynomials it follows that the polynomials
\begin{equation*}
    e_{\tau}(\lambda_1,\dots ,\lambda_m ) :=\prod_{j=1}^me_{\tau_j}(\lambda_1,\dots ,\lambda_k)
\end{equation*}
with $\tau\in\mathcal{P}$ also form a basis.

Before we describe the other bases, we will explain the equivalence between symmetric polynomials and polynomials in the coefficient of a Hermitian matrix which are invariant under conjugation by unitary matrices.
We will denote the space of $m\times m$ Hermitian matrices by $\mathrm{Her}(m)$.
We can see a symmetric orthogonal polynomial $P$ as a function in the coefficients of a $m\times m$ Hermitian matrix as follows
\begin{equation*}
    P(\Lambda ):=P(\lambda_1,\dots ,\lambda_m) ,
\end{equation*}
where $\Lambda$ is a $m\times m$ Hermitian matrix with eigenvalues $\lambda_1 ,\dots ,\lambda_m$.
Note that $P$ is invariant under conjugation by unitary matrices.
In fact, this is a polynomial in the coefficients of the Hermitian matrix.
Indeed, we can extend $P$, as a symmetric polynomial, in terms of the elementary symmetric functions.
By the \emph{Newton identities} the elementary symmetric functions can be seen as polynomials in the coefficients of a Hermitian matrix \cite[\textsection I (2.11')]{macdonald1998symmetric}.
Note that the individual eigenvalues can in general not be written as a polynomial in the coefficients of the matrix, only as an algebraic function.

The other way around we can consider polynomials in the coefficients of a Hermitian matrix invariant under conjugation by unitary matrices.
These can clearly be seen as polynomials in the eigenvalues.
Indeed, a Hermitian matrix can always be diagonalised, i.e. $\mathrm{diag} (\lambda_1,\dots ,\lambda_m) =V\Lambda V^*$ for some $V\in\mathbf{U}(m)$.
The key thing to note here is that the coefficients of $\mathrm{diag}(\lambda_1,\dots ,\lambda_m)$ are precisely the eigenvalues of $\Lambda$.
Note also that these polynomials are automatically symmetric, since one can swap any two eigenvectors in $V$ and obtain another unitary matrix.

We will denote the space of polynomials in the coefficients of a Hermitian matrix invariant under conjugation by unitary matrices by $\mathbf{P}(\mathrm{Her}(m))^{\mathbf{U}(m)}$.
We have just shown that there is an isomorphism
\begin{equation*}
    P(\mathbb{R}^m)^{S_m}\cong\mathbf{P}(\mathrm{Her}(m))^{\mathbf{U}(m)} .
\end{equation*}
These spaces are graded by homogenous polynomials of the same total degree and it is not too difficult to seen that this isomorphism respects this grading.

Related to the elementary symmetric polynomials are the \emph{complete symmetric polynomials}
\begin{equation}\label{eq:complete-symmetric-polynomial}
    h_r(\lambda_1,\dots ,\lambda_m) :=\sum_{\ell_1 +\dots +\ell_m =r}\lambda_1^{\ell_1}\cdots\lambda_m^{\ell_m} .
\end{equation}
The products
\begin{equation*}
    h_{\tau}(\lambda_1,\dots ,\lambda_m ) :=\prod_{j=1}^mh_{\tau_j}(\lambda_1,\dots ,\lambda_k)
\end{equation*}
with $\tau\in\mathcal{P}$ also form a basis.

Yet another basis is given by the \emph{power sum symmetric polynomials}, which are defined by
\begin{equation*}
    PS_{\tau}(\Lambda ) 
    :=\prod_{j=1}^m\mathrm{Tr}(\Lambda^{\tau_j})
    =\prod_{j=1}^m\sum_{\ell =1}^m\lambda^{\tau_j}_{\ell}
\end{equation*}
for some partition $\tau$.
The absolute degree of $PS_{\tau}$ is $(|\tau|,0,\dots ,0)$, but it is in general not a homogeneous polynomial.
For a proof that these polynomials form a basis see \cite[\textsection I (2.12)]{macdonald1998symmetric}.
The power sum symmetric polynomials can for example be used to define a product on the symmetric polynomials, called the \emph{Frobenius} or sometimes the \emph{Kronecker product} \cite{REMMEL1989100}.
They arguably have the simplest expression in terms of the coefficients of a Hermitian matrix.

Finally, we will describe a basis which will be particularly useful in the proof of some results in chapter \ref{chap:invariant-polynomials}.
\begin{definition}
    The \emph{Schur polynomials} are defined by
    \begin{equation*}
        s_{\tau}(\lambda_1 ,\dots ,\lambda_m ):=\frac{\det\left((\lambda_{j}^{\tau_j +m-\ell })_{1\leq\ell ,j\leq m}\right)}{\prod_{1\leq\ell <j\leq m}(\lambda_{\ell} -\lambda_j)} ,
    \end{equation*}
    for a partition $\tau =(\tau_1,\dots ,\tau_m)$.
    The polynomial $s_{\tau}$ has absolute degree $\tau$ as can be seen from the expansion \cite[(0.6)]{Macdonald1992}.
\end{definition}

For a proof that they form a basis, see for example \cite[\textsection I (3.2)]{macdonald1998symmetric}.

By \cite[Theorem 4.1.1.]{kulikauskas1993symmetric} the Schur functions can be written in terms of the coefficients of a Hermitian matrix as follows
\begin{align}\label{eq:Schur-polynomials-matrix-coefficients}
    s_{\tau}(\Lambda )=\frac{\det ((\Lambda^{\tau_{\ell} +m-\ell })_{jj})_{1\leq \ell,j\leq m}}{\det ((\Lambda^{m-\ell })_{jj})_{1\leq\ell ,j\leq m}}
\end{align}
for a partition $\tau$, see also \cite{BROIDO199621}.
The proof of this theorem also shows that it is a polynomials in the coefficients of the matrix, see also proposition \ref{prop:kulikauskas-formula}.
We refer to \cite{kulikauskas1993symmetric} for expressions of the monomial symmetric functions and complete symmetric polynomials in terms of the coefficients of a Hermitian matrix.

\begin{remark}
We can extend functions of a Hermitian matrix argument invariant under conjugation by unitary matrices to general complex matrices by seeing them as the corresponding symmetric functions in the eigenvalues, this is equivalent to expanding the polynomial in the power sum symmetric polynomials and evaluating the corresponding trace expression for general matrices.
Note that in this extension we go from real valued functions to complex valued ones.
However, we don't have the equivalence $P(\mathbb{C}^m)^{S_m}\cong\mathbf{P}(\mathbb{C}^{m\times m})^{\mathbf{U}(m)}$, since some (invertible) complex matrices cannot be diagonalised.
\end{remark}

The Schur polynomials are a scalar multiple of the zonal polynomials with respect to the Gelfand pair $(Gl_m(\mathbb{C}) ,\mathbf{U}(m))$ \cite[\textsection VII 5.]{macdonald1998symmetric}.
We will now briefly recall what this means.
The complex general linear group $Gl_m(\mathbb{C})$ induces a representation on $\mathbf{P}_{r}(\mathrm{Her}(m))$ by the congruence transform $\Lambda\mapsto g\Lambda g^*$ for $r\in\mathbb{N}$.
We then have the following decomposition into irreducible subrepresentations
\begin{equation}\label{eq:single-variable-irreducible-decomposition}
    \mathbf{P}_r(\mathrm{Her}(m)) =\bigoplus_{\tau\vdash r}V_{\tau}\otimes\overline{V}_{\tau},
\end{equation}
where $V_{\tau}$ is some representation isomorphic to the irreducible invariant representations of $Gl_m(\mathbb{C})$ corresponding to a partition $\tau$ and $\overline{V}_{\tau}$ is its conjugate representation, see for example \cite[\textsection I A 8.]{macdonald1998symmetric}.
The subrepresentation $V_{\tau}\otimes\overline{V}_{\tau}$ has a $1$-dimensional subspace invariant under the action of $Gl_m(\mathbb{C})$ restricted to $\mathbf{U}(m)$.
The Schur polynomial $s_{\tau}$ spans this subspace.

In particular, they have the following properties, which will be important for us:
\begin{equation}\label{eq:integral-formulas-Schur-functions}
\begin{split}
    \int_{\mathbf{U}(m)}s_{\tau}(AUBU^*)dU =&\frac{s_{\tau}(A)s_{\tau}(B)}{s_{\tau}(I_m)},\\
    \int_{\mathbf{U}(m)}s_{\tau}(AU)s_{\rho}(BU^*)dU =&\delta_{\tau\rho}\frac{s_{\tau}(AB)}{s_{\tau} (I_m)}
    \end{split}
\end{equation}
for $A,B\in\mathfrak{gl}_m(\mathbb{C})$ and $\tau,\rho\in\mathcal{P}$, see also \cite[Lemma 2.1.]{novak2022}.
  
\subsection{Invariant orthogonal polynomials}\label{subsec:invariant-orthogonal-polynomials}
In this subsection, we describe the basics of the theory of orthogonal invariant polynomials in a more general setting.
Let $G$ be a Lie group acting smoothly on some finite dimensional vector space $V$.
The vector space of polynomials on $V$ invariant under $G$ will be denoted by $\mathbf{P}(V)^G$.
In the next subsection we will specialise to $\mathbf{P}(V)^G=\mathbf{P}(\mathrm{Her}(m))^{\mathbf{U}(n)}$.
Let $\mu$ be some compactly supported measure on $V$ invariant\footnote{A measure $\mu$ on $V$ is called \emph{invariant} under $G$ if $gA:=\{ gv\mid v\in A\}$ is measurable and $\mu (gA)=\mu (A)$ for measurable $A\subseteq V$.} under $G$.
To ensure that the inner product given by integration with respect to $\mu$ is nondegenerate, we will assume that $\mu$ is supported on a set with non-empty interior.
Let $\Theta$ be a graded index set partially ordered by some partial order $\leq$, we will denote the grade of $\theta\in\Theta$ by $|\theta |\in\mathbb{N}$, such that $\theta\leq\rho$ implies $|\theta|\leq |\rho|$ and there exists some basis $(b_{\theta})_{\theta\in\Theta}$ of $\mathbf{P}(V)^G$ such that $b_{\theta}$ has total degree $|\theta |$.
We will call such a basis $(b_{\theta})_{\theta\in\Theta}$ a \emph{graded basis}.
We will denote the set of homogeneous polynomials in $\mathbf{P}(V)^G$ of total degree $r$ by $\mathbf{P}_r(V)^G$.

Polynomials $(P_{\theta})_{\theta\in\Theta}$ are called \emph{orthogonal} with respect to $\mu$ if the total degree of $P_{\theta}$ is $|\theta |$ and
\begin{equation*}
    \int_{V} P_{\theta}(x) P_{\rho}(x)\mu (dx) =\delta_{\theta\rho}h_{\theta}\textrm{ for all }\theta ,\rho\in\Theta .
\end{equation*}
By a dimension argument a system of invariant orthogonal polynomials $(P_{\theta})_{\theta\in\Theta ,|\theta|\leq r}$ forms a graded basis of $\bigoplus_{j=1}^r\mathbf{P}_j(V )^{G}$.
In particular, this means that $P_{\theta}$ is orthogonal to all polynomials of total degree strictly less than $|\theta |$.

A linear operator $\mathcal{L}:P(V)^G\rightarrow P(V)^G$ on $V$ invariant under $G$ is called \emph{symmetric} with respect to $\mu$ if
\begin{equation*}
    \int (\mathcal{L}f)gd\mu =\int f(\mathcal{L}g)d\mu
\end{equation*}
for all polynomials $f,g\in \mathbf{P}(V)^G$.
We will also say that the measure $\mu$ is the \emph{symmetric and invariant} measure of the operator $\mathcal{L}$.
Note that, in many cases of interest, there exist multiple symmetric linear operators for a given (probability) measure.
Furthermore, we call $\mathcal{L}$ \emph{polynomial} if it maps $\mathbf{P}_r(V)^G$ into itself for all $r\in\mathbb{N}$.
Polynomial operators are quite rare, see for example \cite{Songzi-Li2019} for the case when $G$ acts trivially on $V$.

The eigenpolynomials of $\mathcal{L}$ with different eigenvalues are easily seen to be orthogonal.
Indeed, if $f,g$ are eigenfunctions of $\mathcal{L}$ with eigenvalues $\lambda\neq\nu$ respectively, then
\begin{equation*}
    \lambda\int fgd\mu =\int (\mathcal{L}f)gd\mu =\int f(\mathcal{L}g)d\mu =\nu\int fgd\mu ,
\end{equation*}
which shows that the integral $\int fgd\mu$ has to be zero.
Note that the eigenfunctions with the same eigenvalues form a vector space, so we can always find orthogonal eigenfunctions if the eigenvalues are the same using the Gram--Schmidt algorithm.

The other implication does not hold in general, for this we need the following definitions, see \cite[Definition 1.3.4.]{HECKMAN1995}.

\begin{definition}
    Let $(b_{\theta})_{\theta\in\Theta}$ be a graded basis of $\mathbf{P}(V)^{G}$.
    A linear operator $\mathcal{L}:\mathbf{P}(V)^{G}\rightarrow \mathbf{P}(V)^{G}$ is called \emph{triangular} with respect to the ordered basis $(b_{\tau})_{\tau\in\mathcal{P}}$ if
    \begin{equation}\label{eq:triangular-operator}
        \mathcal{L}b_{\theta} =\sum_{\rho\leq\theta} a_{\theta\rho}b_{\rho}
    \end{equation}
    for all $\theta\in\Theta$.
    Furthermore, a polynomial $P\in\mathbf{P}(V)^{G}$ is called \emph{triangular} and of \emph{degree} $\theta$ with respect to the partially ordered basis $(b_{\rho})_{\rho\in\Theta}$ if
    \begin{equation*}
        P(x )=\sum_{\rho\leq\theta} c_{\theta\rho}b_{\rho}(x)\textrm{ with } c_{\theta\theta}\neq 0.
    \end{equation*}
    If $c_{\theta\theta} =1$, we will call the triangular polynomial \emph{monic}.
\end{definition}

Note that the set of partitions with respect to the natural ordering form a lattice.
In particular, any two elements have an upper bound (partitions of the form $(r,0,\dots ,0)$ are for example larger than all paritions $\tau$ such that $|\tau |\leq r$).
This means that every symmetric polynomials would be triangular with respect to any basis ordered by the natural ordering if we drop the condition $c_{\theta\theta}\neq 0$.

Note also that saying that the polynomials in one basis are (monic) triangular in terms of another basis is the same as saying that the transition matrix is lower diagonal (with one's on the diagonal) with respect to any linear ordering extending the partial ordering.
Since the inverse of a lower triangular matrix (with one's on the diagonal) is again a lower triangular matrix (with one's on the diagonal), being a basis which is (monic) triangular with respect to another basis is an equivalence relation.
Triangularity depends on the basis up to this equivalence.

\begin{example}
In this example the partitions are ordered with respect to the natural ordering.
The Schur polynomials are monic triangularly equivalent to the symmetric monomial polynomials \cite[Corollary 7.10.6]{Stanley1999}.
The corresponding coefficients are called the \emph{Kostka numbers}.
See also \cite[\textsection I 6.]{macdonald1998symmetric} for more relations between the different bases.
\end{example}

The importance of the concept of triangularity is explained by the following proposition, which is \cite[Proposition 1.3.5.]{HECKMAN1995} in a different setting.

\begin{proposition}\label{prop:triangularity-eigenpolynomials}
    Let $(b_{\theta})_{\theta\in\Theta}$ be a graded basis of $\mathbf{P}(V)^{G}$, $\mu$ a measure on $V$ invariant under $G$, and $\mathcal{L}$ an operator symmetric with respect to $\mu$ and triangular with respect to $(b_{\theta})_{\tau\in\mathcal{P}}$.
    Assume that $P$ is a triangular symmetric polynomial of degree $\theta$ such that
    \begin{equation*}
        \int_{V}P(x )b_{\rho}(x )\mu (dx )=0\textrm{ for all }\rho <\theta .
    \end{equation*}
    Then, $P$ is a eigenfunction of $\mathcal{L}$ with eigenvalue $a_{\theta\theta}$, where $a_{\theta\theta}$ is the leading term in \eqref{eq:triangular-operator}.
\end{proposition}
\begin{proof}
Assume
\begin{equation*}
    \mathcal{L}b_{\rho'} =\sum_{\rho\leq\rho'} a_{\rho'\rho}b_{\rho}\,\textrm{ and }\,
    P(x )=\sum_{\rho\leq\theta} c_{\theta\rho}b_{\rho}(x) .
\end{equation*}
We therefore have
\begin{equation}\label{eq:expansion-LP}
    \mathcal{L}P=\sum_{\rho\leq\theta} c_{\theta\rho}\mathcal{L}b_{\rho}
    =\sum_{\rho'\leq\rho\leq\theta} c_{\theta\rho}a_{\rho\rho'}b_{\rho'} .
\end{equation}
Because $\mathcal{L}$ is symmetric, we also obtain
\begin{equation*}
    \int_{V} (\mathcal{L}P)b_{\rho}d\mu =
    \int_{V} P(\mathcal{L}b_{\rho})d\mu
    =\sum_{\rho'\leq\rho}a_{\rho\rho'}\int_{V} Pb_{\rho'}d\mu
    =0
\end{equation*}
for $\rho <\theta$.
Taken together we see
\begin{equation*}
    \mathcal{L}P\in \mathrm{span}_{\mathbb{C}}\{ b_{\rho}\mid\rho\leq\theta\}\cap (\mathrm{span}_{\mathbb{C}}\{ b_{\rho}\mid\rho <\theta\})^{\perp} =\mathrm{span}_{\mathbb{C}}\{ P\},
\end{equation*}
where $\perp$ denotes the orthogonal complement with respect to the inner product given by integration with respect to $\mu$.
The second equality follows from a dimension argument together with the assumptions on $P$.
We see that $\mathcal{L}P=aP$ for some $a\in\mathbb{C}$.
That $a=a_{\theta\theta}$ follows directly from the expansion \eqref{eq:expansion-LP}.
\end{proof}

\begin{corollary}\label{cor:eigenfunctions-orthogonal-polynomials}
    Let $(P_{\theta})_{\theta\in\Theta}$ be orthogonal polynomials with respect to some invariant measure $\mu$.
    Assume that $P_{\rho}$ is triangular and of degree $\rho$ with respect to some partially order basis $(b_{\tau})_{\tau\in\mathcal{P}}$ for $\rho\in\Theta$, then
    \begin{equation}\label{eq:orthogonal-with-respect-to-basis}
        \int_{V}P_{\theta}(x )b_{\rho}(x )\mu (d\Lambda )=0\textrm{ for all }\rho <\tau\in\mathcal{P} .
    \end{equation}
    In particular, $(P_{\theta})_{\theta\in\Theta}$ forms a basis of simultaneous eigenfunctions of operators which are symmetric with respect to $\mu$ and triangular with respect to $(b_{\theta})_{\theta\in\Theta}$.
\end{corollary}

The reverse statement, i.e. that any set of triangular polynomials $(P_{\theta})_{\theta\in\Theta}$ satisfying \eqref{eq:orthogonal-with-respect-to-basis}, will be orthogonal polynomials, will not hold in general.
Indeed, the condition \eqref{eq:orthogonal-with-respect-to-basis} will be an empty condition if the partial order is trivial.

A collection $(P_{\theta})_{\theta\in\Theta}$ of triangular polynomials  satisfying \eqref{eq:orthogonal-with-respect-to-basis} can be obtained using the Gram--Schmidt orthogonalisation process with respect to a basis $(b_{\theta})_{\theta\in\Theta}$.
One possible technique for proving pairwise orthogonality is to show that the commuting algebra of triangular and invariant operators is large enough to separate points, this is the basis of the theory outlined in \cite{HECKMAN1995} for symmetric spaces.
However, there does not seem to be much literature about the Gram--Schmidt process in this general setting.

\subsection{Hermitian Jacobi polynomials}
We will now return to the case $G=\mathbf{U}(m)$, which works on $V=\mathrm{Her}(m)$ by conjugation.
Recall that $\Lambda$ denotes a $m\times m$ Hermitian matrix with eigenvalues $\lambda_1 ,\dots ,\lambda_m$.
In this case the graded index set is given by the partitions $\mathcal{P}$ with grading the weight of the partition and the partial ordering is the graded natural ordering.
A linear operator $\mathcal{L}$ with respect to some invariant measure $\mu$ can be seen both as a linear function between symmetric polynomials and polynomials in matrix coefficients invariant under conjugation by unitary matrices.
In other words it can be expressed both in terms of the eigenvalues and in terms of the matrix coefficients.

We represent the Lebesgue measure on the coefficients of a complex matrix by $d\Lambda$ and the normalised Haar measure on $\mathbf{U}(m)$ by $dU$.
The Jacobian of the transformation from the coefficients of the matrix to the eigenvalues is given by the square of the Vandermonde determinant, i.e.
\begin{equation*}
    d\Lambda =\frac{\pi^{\frac{n(n-1)}{2}}}{\prod_{j=1}^m j!}dU\prod_{1\leq\ell <j\leq m} (\lambda_{\ell} -\lambda_j)^2d\lambda_1\dots d\lambda_m .
\end{equation*}
This is a special case of the Weyl integration formula, see \cite[Theorem 10.1.4]{Faraut2008-ov}.
When integrating invariant functions the $dU$ term drops out.

We let $\kappa_1,\kappa_2 >m/2-1$ be fixed.
In this section we will consider a basis of symmetric polynomials orthogonal with respect to the \emph{complex matrix-variate Beta distribution}
\begin{equation}\label{eq:complex-matrix-variate-Beta-distribution}
    \begin{split}
        B_m^{(\kappa_1 ,\kappa_2)}(d\Lambda) :=&C_{\kappa_1, \kappa_2}\det (\Lambda )^{\kappa_1 -m/2}\det (I-\Lambda)^{\kappa_2 -m/2}d\Lambda \\
        =&C_{\kappa_1 ,\kappa_2}\left(\prod_{1\leq\ell <j\leq m}(\lambda_{\ell} -\lambda_j)^2\right)\prod_{j=1}^m \lambda_j^{\kappa_1 -m/2}(1-\lambda_j)^{\kappa_2-m/2} d\lambda_j ,
    \end{split}
\end{equation}
where $C_{\kappa_1 ,\kappa_2} >0$ is a normalisation constant.
These polynomials will serve as a good example of invariant polynomials of a Hermitian matrix argument.
Furthermore, the main aim of this paper is to describe a multivariate matrix extension of these polynomials.

\begin{definition}
    The \emph{Hermitian Jacobi polynomials} are defined by
    \begin{equation}\label{eq:def-HO-polynomials}
        P_{\tau}^{(\kappa_1,\kappa_2 ,2)}(\Lambda ):=\frac{\det\left( (P^{(\kappa_1 -m/2,\kappa_2 -m/2)}_{\tau_{\ell }+m-\ell } (\lambda_{j} ))_{1\leq \ell ,j\leq m}\right)}{\prod_{1\leq \ell < j\leq m} (\lambda_{\ell} -\lambda_j)} ,
    \end{equation}
    for $\tau\in\mathcal{P}$, where $P^{(\alpha ,\beta)}_j$ is the $j^{\textrm{th}}$-Jacobi polynomial of index $(\alpha ,\beta )$.
\end{definition}

The Hermitian Jacobi polynomials are a special case of the \emph{Heckman--Opdam polynomials of type BC} also sometimes referred to as the \emph{multivariate Jacobi polynomials}.

A special case of \cite[Theorem 3.1.]{BALDERRAMA2005486} shows that the Hermitian Jacobi polynomials are orthogonal with respect to the weight $B_m^{(\kappa_1 ,\kappa_2)}$ and by multilinearity of the determinant they can be shown to be triangular with respect to the Schur polynomials and thus the symmetric monomial polynomials, see also \cite[Proposition 12.3.1]{Faraut2008-ov}.
In particular, the Hermitian Jacobi polynomials are obtained by applying the Gram--Schmidt orthogonalisation process to the Schur functions with respect to the complex matrix-variate Beta distribution. 

A symmetric linear operator with respect to the complex matrix-variate Beta distribution $B_m^{(\kappa_1 ,\kappa_2)}$ is given by
\begin{equation*}
\begin{split}
    \mathcal{L}_{(\kappa_1,\kappa_2)}^m =&4\sum_{j=1}^m\lambda_j(1-\lambda_j) \frac{\partial^2}{\partial\lambda_j^2} \\
    &-2\sum_{j=1}^m\left(\left(\kappa_1-\frac{m}{2}\right)
    -\left( \kappa_1 +\kappa_2-m\right)\lambda_j +2\sum_{\ell =1;\ell\neq j}^m\frac{\lambda_j (1-\lambda_{\ell}) +\lambda_{\ell} (1-\lambda_j)}{\lambda_{j} -\lambda_{\ell}}\right)\frac{\partial}{\partial\lambda_j} ,
\end{split}
\end{equation*}
in terms of the eigenvalues, see \cite[Lemma 8.1.10.]{Baudoin2024-is}.
In terms of the matrix coefficients this operator takes the form
\begin{equation}\label{eq:HO-operator}
\begin{split}
    \mathcal{L}_{(\kappa_1 ,\kappa_2 )}^m =& \frac{1}{2}\sum_{\alpha ,\beta ,\gamma ,\delta =1}^m ((I_m-\Lambda )_{\alpha \delta}\Lambda_{\gamma\beta} +(I_m-\Lambda )_{\gamma\beta}\Lambda_{\alpha \delta})\frac{\partial^2}{\partial\Lambda_{\gamma\delta}\partial\Lambda_{\alpha\beta}} \\ 
    &+\sum_{\alpha ,\beta =1}^m\left(\left(\kappa_1 +\frac{m}{2}\right) \delta_{\alpha\beta}-\left(\kappa_1 +\kappa_2 +m\right)\Lambda_{\alpha\beta}\right)\frac{\partial}{\partial\Lambda_{\alpha\beta}} .
\end{split}
\end{equation}
Note that the two terms in the first sum are the same up to replacing $(\alpha ,\delta)$ by $(\gamma ,\beta )$ and every term thus appears twice in the sum.
The operator $\mathcal{L}$ is triangular with respect to the monomial symmetric polynomials \cite[Proposition 1.3.10.]{HECKMAN1995}, and thus with respect to the Schur polynomials.

In particular, the Hermitian Jacobi polynomials of index $(\kappa_1,\kappa_2 )$ are eigenfunctions of the operator $\mathcal{L}_{(\kappa_1 ,\kappa_2)}$ \cite[Proposition 1.3.5.]{HECKMAN1995}.
The eigenvalue of $\mathcal{L}_{(\kappa_1,\kappa_2)}$ corresponding to $P_{\tau}^{(\kappa_1,\kappa_2 ,2)}$ is equal to
\begin{equation*}
    r_{\tau}^{(\kappa_1 ,\kappa_2)} :=-\sum_{j=1}^m\tau_j\left(\tau_j +\kappa_1 +\kappa_2 +1+m - 2j\right)\textrm{ for }\tau\in\mathcal{P},
\end{equation*}
see for example \cite[\textsection B.4]{Baudoin2024-is}.
We will however need a slightly stronger result later on.

\begin{proposition}\label{prop:identity-on-irreducible-components}
Let $V_{\tau},\overline{V}_{\tau}$ be as in equation \eqref{eq:single-variable-irreducible-decomposition}.
The part of $\mathcal{L}_{(\kappa_1,\kappa_2)}$ which preserves the total degree acts as the identity times $r_{\tau}^{(\kappa_1 ,\kappa_2)}$ on $V_{\tau_j}\otimes\overline{V}_{\tau_j}$.
\end{proposition}
\begin{proof}
    Since we are working with a graded ordering, we only need to look at the parts of $\mathcal{L}^m_{\alpha_j ,\beta_j}$ which do not lower the degree.
    These are given by
    \begin{equation*}
        \mathcal{C}_0:=-\sum_{a,b ,c,d =1}^m\Lambda_{ad}\Lambda_{cb}\frac{\partial^2}{\partial\Lambda_{cd}\partial\Lambda_{ab}} -(\alpha_j +\beta_j +m)\sum_{a,b=1}^m\Lambda_{ab}\frac{\partial}{\partial\Lambda_{ab}} .
    \end{equation*}
    We will show that this operator works as a multiple of the identity on $V_{\tau}\otimes\overline{V}_{\tau}$.
    The operator $\mathcal{C}_0$ is closely related to the Casimir operator on $\mathfrak{gl}_m(\mathbb{C})$, but we will use a more elementary approach in this proof.
    We have the \emph{Euler vector field} part
    \begin{equation*}
        \mathcal{V}:=\sum_{a,b =1}^m\Lambda_{ab}\frac{\partial}{\partial\Lambda_{ab}}  .
    \end{equation*}
    The following standard argument shows that all homogeneous functions are eigenfunctions of $\mathcal{V}$.
    If $f$ is homogeneous of degree $r$ in the coefficients of $\Lambda$, i.e.
    \begin{equation*}
        g(\mu):=f(\mu\Lambda )-\mu^{r}f(\Lambda ) =0,
    \end{equation*}
    then the chain rule gives
    \begin{equation*}
        0=\frac{d}{d\mu}g(\mu )=\mathcal{V}f(\mu\Lambda ) -r\mu^{r-1}f(\Lambda) .
    \end{equation*}
    For $\mu =1$ this gives $\mathcal{V}f=rf$.
    Since the space $V_{\tau}\otimes\overline{V}_{\tau}$ consists of homogeneous polynomials of degree $|\tau|$, $\mathcal{V}_{|V_{\tau}\otimes\overline{V}_{\tau}}=|\tau|\mathrm{Id}_{V_{\tau}\otimes\overline{V}_{\tau}}$.

    The quadratic term
    \begin{equation*}
        \mathcal{Q} :=\sum_{a,b ,c,d =1}^m\Lambda_{ad}\Lambda_{cb}\frac{\partial^2}{\partial\Lambda_{cd}\partial\Lambda_{ab}}
    \end{equation*}
    is an intertwiner of the $Gl_m(\mathbb{C})$ action on $\mathbf{P}_{r}(\mathrm{Her}(m))$ induced by $\Lambda\mapsto g\Lambda g^*$.
    Note that this uses that the operator preserves the total degree.
    To see that it is a $Gl_m(\mathbb{C})$-intertwiner, let $g\in Gl_{m}(\mathbb{C})$, then by the chain rule
    \begin{equation*}
        \mathcal{Q}F(g\Lambda g^*) =\sum_{a,b ,c,d =1}^m\sum_{\alpha,\beta,\gamma ,\delta =1}^mg_{\alpha a}\Lambda_{ad}(g^*)_{d\delta}g_{\gamma c}\Lambda_{cb}(g^*)_{b\beta}(\partial_{\gamma\delta}\partial_{\alpha\beta}F)(g\Lambda g^*) .
    \end{equation*}
    for $F\in P_{r}(\mathrm{Her}(m))$.
    Let $\Pi_{\rho}:\mathbf{P}_{r_j}(\mathrm{Her}(m))\rightarrow V_{\rho}\otimes\overline{V}_{\rho}$ be the projection onto the components of the decomposition into irreducible subspaces, which is an $Gl_m(\mathbb{C})$-intertwiner.
    By Schur's lemma $\Pi_{\rho }\circ\mathcal{Q}_{|V_{\tau}\otimes\overline{V}_{\tau}}$ is zero if $\rho\neq\tau$.
    This shows that $\mathcal{Q}$ maps $V_{\tau}\otimes\overline{V}_{\tau}$ into itself.
    Another application of Schur's lemma gives that the operator $\mathcal{Q}_{|V_{\tau}\otimes\overline{V}_{\tau}}=\Pi_{\tau }\circ\mathcal{Q}_{|V_{\tau}\otimes\overline{V}_{\tau}}$ acts as a scalar multiple of the identity on $V_{\tau_j}\otimes\overline{V}_{\tau_j}$.

    We see that the whole operator $\mathcal{C}_0$ acts as a scalar multiple on $V_{\tau_j}\otimes\overline{V}_{\tau_j}$, and this scalar has to be equal to $r_{\tau}^{\kappa}$ since the leading term of $P_{\tau}^{(\kappa_1,\kappa_2 ,2)}$ lies in $V_{\tau_j}\otimes\overline{V}_{\tau_j}$.
\end{proof}

Next we give an explicit expression for the Hermitian Jacobi polynomials in terms of the matrix coefficients.
Our proof relies heavily on the proof of \cite[Theorem 4.1.1.]{kulikauskas1993symmetric}.
We will first need a version of the Jacobi--Trudi identity for the Hermitian Jacobi polynomials, which is well known for the Schur polynomials \cite[Equation (3.4)]{macdonald1998symmetric}.

\begin{lemma}\label{lemma:Jacobi-Trudi-for-HO-polynomials}
    Let $h_r$ denote the complete symmetric polynomial of degree $r$, defined in equation \eqref{eq:complete-symmetric-polynomial}.
    Then, for a partition $\tau$, we have
    \begin{equation*}
        P_{\tau}^{(\kappa_1 ,\kappa_2 ,2)}(\Lambda) =\det\left(\left( P_{\tau_{\ell} +m -\ell}^{(\kappa_1 -m/2 ,\kappa_2 -m/2)} (h_{\cdot -\ell +j}(\Lambda))\right)_{1\leq\ell ,j\leq m}\right) ,
    \end{equation*}
    where we use the evaluation $(h_{\cdot +\ell -j}(\Lambda))^a =h_{a+\ell -j}(\Lambda)$.
\end{lemma}
\begin{proof}
    Let
    \begin{equation*}
        e^{(j)}_r(\Lambda) =\sum_{\substack{\ell_1 <\dots <\ell_{r-1} \\ \ell_1 ,\dots ,\ell_{r-1}\neq j}}\lambda_{\ell_1}\cdots\lambda_{\ell_{r-1}}
    \end{equation*}
    be the elementary symmetric polynomial of degree $r$ without the $j^{\textrm{th}}$ variable.
    Define
    \begin{equation*}
        Q_{\alpha} :=\left( P_{\alpha_{\ell} }^{(\kappa_1 -m/2 ,\kappa_2 -m/2)}(\lambda_j)\right)_{1\leq\ell ,j\leq m},\quad
        \tilde{H}_{\alpha} :=\left( P_{\alpha_{\ell} }^{(\kappa_1 -m/2 ,\kappa_2 -m/2)} (h_{\cdot -m+j}(\Lambda))\right)_{1\leq\ell ,j\leq m}
    \end{equation*}
    and
    \begin{equation*}
        M:=\left( (-1)^{m-\ell }e_{m-\ell }^{(j)}\right)_{1\leq\ell ,j\leq m} .
    \end{equation*}

    The components of \cite[Theorem (3.6)]{macdonald1998symmetric} are given by
    \begin{equation*}
        \sum_{i=1}^m h_{a-m+i}(\Lambda)(-1)^{m-i}e^{(j)}_{m-i}(\Lambda ) =\lambda_j^{a} .
    \end{equation*}
    We then have
    \begin{equation}\label{eq:matrix-Jacobi-Trudy}
        \tilde{H}_{\alpha}M =Q_{\alpha} .
    \end{equation}
    Using that $\det (M)=\prod_{1\leq\ell <j\leq m}(\lambda_{\ell}-\lambda_{j})$ and taking determinants in \eqref{eq:matrix-Jacobi-Trudy} gives
    \begin{equation*}
        \det{ (\tilde{H}_{\alpha})} \prod_{1\leq\ell <j\leq m}(\lambda_{\ell}-\lambda_{j}) =\det (Q_{\alpha})
    \end{equation*}
    and the lemma follows by taking $\alpha_{\ell} =\tau_{\ell} +m-\ell$ for $1\leq\ell\leq m$ and dividing by $M$.
\end{proof}

\begin{proposition}\label{prop:kulikauskas-formula}
    For a partition $\tau =(\tau_1 ,\dots ,\tau_m)$ we have
    \begin{equation*}
        P_{\tau}^{(\kappa_1 ,\kappa_2 ,2)}(\Lambda) =\frac{\det\left(\left( P_{\tau_{\ell} +m-\ell}^{(\kappa_1 -m/2 ,\kappa_2 -m/2)} (\Lambda )\right)_{jj}\right)_{1\leq\ell ,j\leq m}}{\det\left( (\Lambda^{m-\ell})_{jj}\right)_{1\leq\ell ,j\leq m}} .
    \end{equation*}
\end{proposition}
\begin{proof}
    Let $h_r$ denote the complete symmetric polynomial of degree $r$.
    The first result of the proof of \cite[Theorem 4.1.1.]{kulikauskas1993symmetric} we will use is
    \begin{equation}\label{eq:Kulikauskas-equation}
        \sum_{i=1}^m h_{a-m+i}(\Lambda)(-1)^{m-i}\tilde{e}^{(j)}_{m-i}(\Lambda ) =(\Lambda^{a})_{jj} ,
    \end{equation}
    where $\tilde{e}^{(j)}_{r}$ is the generating function of boxes of cycles with sign that employ $r$ edges but do not employ the letter $j$.
    The precise definition of $\tilde{e}^{(j)}_{r}$ is not that important for us, just that there exists some functions satisfying the above equation, so we will not further elaborate on the precise definition of $\tilde{e}^{(j)}_{r}$.
    Define the matrices $Q_{\alpha}$ and $\tilde{H}_{\alpha}$ as in the proof of the previous lemma and the matrix
    \begin{equation*}
        \tilde{M}:=\left( (-1)^{m-\ell }\tilde{e}_{m-\ell }^{(j)}(\Lambda)\right)_{1\leq\ell ,j\leq m} .
    \end{equation*}
    The relation \eqref{eq:Kulikauskas-equation} readily gives
    \begin{equation*}
        \tilde{H}_{\alpha}\tilde{M} =Q_{\alpha} .
    \end{equation*}
    Taking the determinant, using that
    \begin{equation*}
        \det (\tilde{M})=\det\left( (\Lambda^{m-\ell})_{jj}\right)_{1\leq\ell ,j\leq m}
    \end{equation*}
    as also shown in the proof of \cite[Theorem 4.1.1.]{kulikauskas1993symmetric}, and lemma \ref{lemma:Jacobi-Trudi-for-HO-polynomials} we see
    \begin{equation*}
        P_{\tau}^{(\kappa_1,\kappa_2 ,2)}(\Lambda)\cdot\det\left( (\Lambda^{m-\ell})_{jj}\right)_{1\leq\ell ,j\leq m}=
        \det\left(\left( P_{\tau_{\ell} +m-\ell}^{(\kappa_1 -m/2 ,\kappa_2 -m/2)} (\Lambda )\right)_{jj}\right)_{1\leq\ell ,j\leq m} .
    \end{equation*}
    Note in particular that this means that $\det\left(\left( P_{\tau_{\ell} +m-\ell}^{(\kappa_1 -m/2 ,\kappa_2 -m/2)} (\Lambda )\right)_{jj}\right)_{1\leq\ell ,j\leq m}$ is divisible by $\det\left( (\Lambda^{m-\ell})_{jj}\right)_{1\leq\ell ,j\leq m}$ as a polynomial in the coefficients of $\Lambda$.
    After rearranging, we can conclude the theorem.
\end{proof}

\begin{remark}
    The discussion here, with the exception of the symmetric linear operator, is easily extended to to the more general class of symmetric polynomials treated in \cite{BALDERRAMA2005486}.
\end{remark}

\section{Invariant polynomials of Hermitian matrix arguments}\label{chap:invariant-polynomials}
The viewpoint of symmetric polynomials as polynomials of the matrix coefficients invariant under conjugation by unitary matrices is particularly interesting because we extend it to multiple Hermitian matrix variables.
In this section we will describe this extension.
The main novel result in this section is a matrix version of Koornwinder's method.

\begin{definition}
    A \emph{multivariate invariant polynomial of Hermitian matrix arguments} is a polynomial $P:\mathrm{Her}(m)^{k}\rightarrow\mathbb{C}$ in the coefficients of the matrices which is invariant under \emph{simultaneous conjugation with unitary matrices}, i.e.
    \begin{align*}
        P(\Lambda_1,\dots ,\Lambda_{k})\mapsto P(U\Lambda_1 U^*,\dots ,U\Lambda_{k}U^*)
    \end{align*}
    for all $U\in\mathbf{U}(m)$.
    We will denote the vector space spanned by all such polynomials which are homogeneous of total degree $r_1,\dots ,r_k$, in the coefficients of the matrices\footnote{Note that the total degree is invariant under simultaneous conjugation with unitary matrices, since it is a linear transformation in the coordinates of the matrix.} $\Lambda_1,\dots ,\Lambda_k$ respectively, by $\mathbf{P}_{r_1,\dots ,r_k}(\mathrm{Her}(m)^k)^{\mathbf{U}(m)}$.
\end{definition}

\subsection{Simultaneous invariants and generators}
This extension is non-trivial, because non-commuting matrices $\Lambda_1,\dots ,\Lambda_{k}$ are not simultaneously diagonalisable and we are therefore not able to write these polynomials in terms of the eigenvalues of these matrices in general.
We will however be able to write them in terms of the \emph{simultaneous invariants} of the $k$ matrices $\Lambda_1,\dots ,\Lambda_k$.
Unfortunately, these simultaneous invariant are rather elusive.
The following examples are essential in our understanding of the simultaneous invariants.

\begin{example}
    A large class of examples of multivariate invariant polynomials of Hermitian matrix arguments are given by
    \begin{equation}\label{eq:products-of-traces}
        \prod_{j=1}^q\mathrm{Tr}(\Lambda_1^{a_{11j}}\Lambda_2^{a_{21j}}\cdots \Lambda_{k}^{a_{k1j}}\cdots \Lambda_1^{a_{1pj}}\Lambda_2^{a_{2pj}} \cdots\Lambda_{k}^{a_{kpj}})
    \end{equation}
    for some $a_{ij\ell}\in\mathbb{N}$ with $1\leq i\leq k$, $1\leq j\leq p$ and $1\leq\ell\leq q$.
    Note that these polynomials have degree $r_1,\dots ,r_k$ in the coefficients of the matrices $\Lambda_1,\dots ,\Lambda_k$ respectively, where $r_i :=\sum_{j=1}^p\sum_{\ell =1}^qa_{ij\ell}$ for $1\leq i\leq k$.

    The simplest non-trivial case is given by
    \begin{equation*}
        P(A,B) :=\mathrm{Tr}(AB)=\sum_{p,q=1}^m A_{pq}B_{qp}=\sum_{p,q=1}^m A_{pq}\overline{B}_{pq} .
    \end{equation*}
\end{example}

Finding generators of spaces like $\mathbf{P}_{r_1,\dots ,r_k}(\mathrm{Her}(m)^k)^{\mathbf{U}(m)}$ has quite a long history and goes back to at least \cite{ARTIN1969532}.
Using polarisation of homogeneous polynomials and the fact that simultaneous invariance by unitary matrices is the same as simultaneous invariance by invertible matrices, we can deduce the following theorem, see \cite[Theorem 11.2. (a)]{PROCESI1976306}.

\begin{theorem}\label{thm:generators-polynomial-Hermitian-matrix-arguments}
    The space $\mathbf{P}_{r_1,\dots ,r_k}(\mathrm{Her}(m)^k)^{\mathbf{U}(m)}$ is spanned by the polynomials of the form \eqref{eq:products-of-traces}.
    In other words the simultaneous invariant are all determined by the trace operator.
\end{theorem}

Note that this means that we can extend polynomials in $\mathbf{P}_{r_1,\dots ,r_k}(\mathrm{Her}(m)^k)^{\mathbf{U}(m)}$ from $\mathrm{Her}(m)^k$ to $(\mathbb{C}^{m\times m})^k$ in a canonical way and this extension is still invariant under simultaneous conjugation with unitary matrices.

\begin{remark}\label{rmk:generators-polynomial-Hermitian-matrix-arguments}
    It can be shown that the relations among polynomials of the form \eqref{eq:products-of-traces} are all a consequence of the fact that the trace is invariant under cyclic permutations and the Cayley--Hamilton theorem, see \cite[\textsection 4]{PROCESI1976306}.
\end{remark}

While these results give a rather complete picture of the invariants, it still seems difficult to explicitly write down a basis or a minimal set of generators of $\mathbf{P}(\mathrm{Her} (m)^{k})^{\mathbf{U}(m)}$ in practice.
However, a finite set of generators is given by
\begin{equation*}
    \mathrm{Tr}(\Lambda_{j_1}\cdots\Lambda_{j_{\ell}})\textrm{ with } \ell\leq 2^m-1,
\end{equation*}
see \cite[Theorem 3.4. (a)]{PROCESI1976306}.
This shows that we need at most $(2^m-1)^{k+1}$ generators.
Note that when $k=1$ we only need $m$ generators, so this is far from a minimal set of generators.

\subsection{Bases and multivariate invariant orthogonal polynomials}
We will now describe a basis of $\mathbf{P}(\mathrm{Her}(m)^k)^{\mathbf{U}(m)}$ using representation theory.
This basis is indexed by an index set $\Theta_k$, which we will use to define multivariate invariant orthogonal polynomials.

This basis is a complex analogue of the \emph{invariant polynomials with matrix arguments}, see \cite{Davis1981, Chikuse1986}.
See also \cite{JoseA2011, nagar2005generalized, Ratnarajah2005} for more properties of these polynomials.
Recall that the Schur polynomials are up to a scalar the zonal polynomials of the Gelfand pair $(Gl_m(\mathbb{C}),\mathbf{U}(m))$.

\begin{definition}\label{def:multivariate-Schur-polynomials}
    The \emph{multivariate Schur polynomials} $(s_{\theta})_{\theta\in\Theta_k}$ are scalar multiples of the \lq zonal polynomials\rq\ of the pair $(Gl_m(\mathbb{C})^k,\Delta (\mathbf{U}(m)))$, where
    \begin{equation*}
        \Delta (\mathbf{U}(m)):=\{(U,\dots ,U)\in Gl_m(\mathbb{C})^k\mid U\in\mathbf{U}(m)\}
    \end{equation*}
    is the diagonal embedding of $\mathbf{U}(m)$ in $Gl_m(\mathbb{C})^k$.
\end{definition}

Note that the pair $(Gl_m(\mathbb{C})^k,\Delta (\mathbf{U}(m)))$ is not always a Gelfand pair when $k>1$.
We will now briefly describe what we mean by a zonal polynomial in this case.
Recall the decomposition \eqref{eq:single-variable-irreducible-decomposition} of $\mathbf{P}_{r_j}(\mathrm{Her}(m))$ into irreducible subrepresentations with respect to the action of $Gl_m(\mathbb{C})$ induced by the congruence transform.
This gives
\begin{equation*}
    \mathbf{P}_{r_1,\dots ,r_k}(\mathrm{Her}(m)^k) 
    =\bigotimes_{j=1}^k\mathbf{P}_{r_j}(\mathrm{Her}(m))
    =\bigoplus_{\substack{\tau_1\vdash r_1,\dots ,\tau_k\vdash r_k}}(V_{\tau_1}\otimes \overline{V}_{\tau_1}\otimes\cdots\otimes V_{\tau_k}\otimes\overline{V}_{\tau_k}) .
\end{equation*}

Fix $\tau_1\vdash r_1,\dots ,\tau_k\vdash r_k$.
$Gl_m(\mathbb{C})$ induces a representation on the Kronecker product $V_{\tau_1}\otimes\cdots\otimes V_{\tau_k}$ by the simultaneous congruence transform $(\Lambda_1,\dots ,\Lambda_k)\mapsto (g\Lambda_1 g^*,\dots ,g\Lambda_k g^*)$.
We can therefore decompose $V_{\tau_1}\otimes\cdots\otimes V_{\tau_k}$ into a direct sum of irreducible representations $V^{\tau_1,\dots ,\tau_k}_{\phi}\cong V_{\phi}$ of $Gl_m(\mathbb{C})$ indexed by partitions $\phi$ with weight $r_1+\cdots +r_k$.
A partition $\phi$ can appear with higher multiplicity in this decomposition.
In fact the multiplicity of $\phi$ is given by $c^{\phi}_{\tau_1,\dots ,\tau_k}$, where $c^{\phi}_{\tau_1,\dots ,\tau_k}$ are the coefficients appearing in the expansion
\begin{equation*}
    s_{\tau_1}(\Lambda )\cdots s_{\tau_k}(\Lambda)=\sum_{\phi\vdash |\tau_1|+\dots +|\tau_k |}c^{\phi}_{\tau_1,\dots ,\tau_k} s_{\phi}(\Lambda) .
\end{equation*}
For $k=2$, these are called the \emph{Littlewood--Richardson coefficients} \cite[\textsection 12.5.1]{Procesi2006-am}, \cite{Collins2020}.
These can be used to inductively obtain the multiplicities for $k>2$.
Similarly one can decompose $\overline{V}_{\tau_1}\otimes\cdots\otimes \overline{V}_{\tau_k}$ into $\overline{V}^{\tau_1,\dots ,\tau_k}_{\psi}$.

We then obtain the direct sum decomposition
\begin{equation*}
    V_{\tau_1}\otimes \overline{V}_{\tau_1}\otimes\cdots\otimes V_{\tau_k}\otimes\overline{V}_{\tau_k} =\bigoplus_{\substack{\phi ,\psi\vdash |\tau_1|+\dots +|\tau_k|}} V^{\tau_1,\dots ,\tau_k}_{\phi}\otimes\overline{V}^{\tau_1,\dots ,\tau_k}_{\psi} .
\end{equation*}
The representation $V^{\tau_1,\dots ,\tau_k}_{\phi}\otimes\overline{V}^{\tau_1,\dots ,\tau_k}_{\psi}$ is irreducible, and contains a one-dimensional subspace $S_{\phi}^{\tau_1,\dots ,\tau_k}$ invariant under conjugation by unitary matrices if and only if $\phi =\psi$ \cite[VII (5.7)]{macdonald1998symmetric}.
Note that subspaces invariant under the action restricted to $\mathbf{U}(m)$ are at most one-dimensional since $(Gl_m(\mathbb{C}),\mathbf{U}(m))$ is a Gelfand pair.

We will use the notation $\phi\in (\tau_1,\dots ,\tau_k)$ to mean that $\phi$ appears in the decomposition of $V_{\tau_1}\otimes\cdots\otimes V_{\tau_k}$ into irreducible representations.
Recall that this implies $\phi\vdash |\tau_1|+\dots +|\tau_k|$.
As noted before, such a partition $\phi$ might appear with multiplicity greater than one for a given $\tau_1,\dots ,\tau_k$, which means that $S^{\tau_1,\dots ,\tau_k}_{\phi}$ is not always uniquely defined.
We will distinguish the partitions corresponding to equivalent irreducible representations with different symbols, call them \emph{equivalent partitions}, and use the notation $\phi\equiv\phi'$ if $\phi$ and $\phi'$ are equivalent partitions for a fixed $k$-tuple of partitions.
However, the direct sum of the invariant subspaces over equivalent partitions
\begin{equation}\label{eq:equivalent-partions}
    \mathcal{S}_{\phi}^{\tau_1,\dots ,\tau_k} :=\bigoplus_{\phi'\equiv\phi} S_{\phi'}^{^{\tau_1,\dots ,\tau_k}}
\end{equation}
is uniquely defined and has dimension $(c_{\tau_1,\dots ,\tau_k}^{\phi})^2$, see \cite{Harris2014} for similar decompositions.

We define
\begin{equation*}
    \Theta_k :=\{(\tau_1,\dots ,\tau_k ;\phi )\in\mathcal{P}^{k+1}\mid\phi\in (\tau_1,\dots ,\tau_k )\} ,
\end{equation*}
where we have to include the partitions $\phi$ with the right multiplicity.
Two indices $(\tau_1, \dots ,\tau_k ;\phi)$, $(\rho_1,\dots ,\rho_k ;\psi)$ are called \emph{equivalent} if $\tau_1=\rho_1,\dots ,\tau_k =\rho_k$ and $\phi\equiv\psi$.
We will equip this index set with the \lq \emph{multigrading}\rq\
\begin{equation*}
    |(\tau_1,\dots ,\tau_k ;\phi )| :=(|\tau_1|,\dots ,|\tau_k|) ,
\end{equation*}
and $\mathbb{N}^k$ with the lexicographical ordering.

The multivariate Schur polynomials are defined as a suitably normalised basis $(s_{\phi}^{\tau_1,\dots ,\tau_k})_{(\tau_1,\dots ,\tau_k ;\phi )\in\Theta_k}$ of $\mathbf{P}(\mathrm{Her}(m)^k)^{\mathbf{U}(m)}$ such that $s_{\phi}^{\tau_1,\dots ,\tau_k}\in\mathcal{S}_{\phi}^{\tau_1,\dots ,\tau_k}$ for all $(\tau_1,\dots ,\tau_k ;\phi )\in\Theta_k$ and satisfying the fundamental relationship
\begin{equation}\label{eq:integral-property-multivariate-Schur}
    \int_{\mathbf{U}(m)}\prod_{j=1}^ks_{\tau_j}(A_jUB_jU^*)dU
    =\sum_{\phi\in (\tau_1,\dots ,\tau_k)}\frac{s_{\phi}^{\tau_1,\dots ,\tau_k}(A_1,\dots ,A_k)s_{\phi}^{\tau_1,\dots ,\tau_k}(B_1,\dots ,B_k)}{s_{\phi}(I_m)}
\end{equation}
for $\tau_1,\dots ,\tau_k\in\mathcal{P}$ and $A_j,B_j\in\mathrm{Her}(m)$ for $1\leq j\leq k$.
While these polynomials are not uniquely defined, they will suffice for our purposes.

The formula \eqref{eq:integral-property-multivariate-Schur} with $A_j =I_m$ for $1\leq j\leq k$ gives
\begin{equation}\label{eq:product-Schur-polynomials}
    \prod_{j=1}^k s_{\tau_j}(\Lambda_j) =\sum_{\phi\in (\tau_1,\dots,\tau_k)} \theta^{\tau_1,\dots ,\tau_k}_{\phi} s^{\tau_1,\dots ,\tau_k}_{\phi}(\Lambda_1,\dots ,\Lambda_k) ,
\end{equation}
where
\begin{equation*}
    \theta^{\tau_1,\dots ,\tau_k}_{\phi} :=\frac{s_{\phi}^{\tau_1,\dots ,\tau_k}(I_m,\dots ,I_m)}{s_{\phi}(I_m)}.
\end{equation*}
We also note that there are analogues of the first integral equation \eqref{eq:integral-formulas-Schur-functions} for multivariate Schur functions, see \cite[Equation (9)]{JoseA2011}, which we will record here for completeness
\begin{equation}\label{eq:integral-formula-multivariate-Schur}
        \int_{\mathbf{U}(m)}s_{\phi}^{\tau_1,\dots ,\tau_k}(A_1UB_1U^*,\dots ,A_kUB_kU^*)dU
        =\frac{s_{\phi}^{\tau_1,\dots ,\tau_k}(A_1,\dots ,A_k)s_{\phi}^{\tau_1,\dots ,\tau_k}(B_1,\dots ,B_k)}{s_{\phi}^{\tau_1,\dots ,\tau_k}(I_m,\dots ,I_m)} .
\end{equation}
We are not aware of analogues of the second formula.

The multivariate Schur polynomials are a graded basis.
Indeed, the multivariate Schur polynomial $s_{\tau_1,\dots ,\tau_k}^\phi(\Lambda_1,\dots ,\Lambda_k)$ has total degree $|\tau_1|,\dots ,|\tau_k|$ in the coefficients of the matrices $\Lambda_1,\dots ,\Lambda_k$ respectively since it is an element of $\mathbf{P}_{r_1,\dots ,r_k}(\mathrm{Her}(m)^k)$.

\begin{example}
    For $m=1$, we have $s_{\tau }(\lambda) =\lambda^{\tau}$ and $s_{\tau_1+\dots +\tau_k}^{\tau_1,\dots ,\tau_k}(\lambda_1,\dots ,\lambda_k) =\lambda_1^{\tau_1}\cdots\lambda_{k}^{\tau_k}$, since the tensor product is simply given by the product of the functions and all functions are invariant under (simultaneous) conjugation by unitary $1\times 1$ matrices.
\end{example}

To calculate multivariate Schur polynomials $s_{\phi}^{\tau_1,\dots ,\tau_k}$ explicitly when $m>1$ we can use the following lemma when $\phi$ is of multiplicity $1$, which is a complex analogue of \cite[(A.4.13)]{Chikuse2003-nh}.

\begin{lemma}[\textbf{Multinomial expansion}]\label{lemma:multinomial-expansion-Schur-function}
    Let $\phi$ be a partition, then
    \begin{equation*}
        s_{\phi}\left(\sum_{j=1}^k\Lambda_j\right) =\sum_{\tau_1,\dots ,\tau_k\in\mathcal{P};\phi\in (\tau_1,\dots ,\tau_k)}\sum_{\phi'\equiv\phi} c^{\tau_1,\dots ,\tau_k}_{\phi'} s^{\tau_1,\dots ,\tau_k}_{\phi'}(\Lambda_1,\dots ,\Lambda_k) .
    \end{equation*}
    for some constants $c^{\tau_1,\dots ,\tau_k}_{\phi}$ depending on the normalisation of the multivariate Schur polynomials.
\end{lemma}
\begin{proof}
    On the one hand we have
    \begin{align*}
        \int_{\mathbf{U}(m)}e^{\mathrm{Tr}\left(\sum_{j=1}^k\Lambda_j UAU\right)}dU
        =&\sum_{\ell =1}^{\infty}\frac{1}{\ell !}\int_{\mathbf{U}(m)}\mathrm{Tr}\left(\sum_{j=1}^{k}\Lambda_j UA U^*\right)^{\ell} \\
        =&\sum_{\ell =1}^{\infty}\sum_{\phi\vdash\ell}\frac{K_{\phi ,(1^{\ell})}}{\ell !} \int_{\mathbf{U}(m)}s_{\phi}\left(\sum_{j=1}^{k}\Lambda_j UAU^*\right) dU\\
        =&\sum_{\ell =1}^{\infty}\sum_{\phi\vdash\ell}\frac{K_{\tau ,(1^{\ell})}}{\ell !}\frac{1}{s_{\phi}(I_m)} s_{\phi}\left(\sum_{j=1}^{k}\Lambda_j \right) s_{\phi}(A),
    \end{align*}
    see \cite[\textsection I proof of (7.6)]{macdonald1998symmetric} for the second equality.
    Note that the Kostka number $K_{\tau ,(1^{\ell})}$ is greater than or equal to $1$ for $|\tau|\vdash\ell$, since $K_{\tau ,\rho}$ is the number of tableaux of shape $\tau$ and weight $\rho$ \cite[\textsection I (6.4)]{macdonald1998symmetric}.
    On the other hand, by equation \eqref{eq:integral-property-multivariate-Schur}, we have
    \begin{multline*}
        \int_{\mathbf{U}(m)}\prod_{j=1}^ke^{\mathrm{Tr}(\Lambda_jUAU^*)}dU
        =\sum_{\ell_1,\dots ,\ell_k =0}^{\infty} \int_{\mathbf{U}(m)}\prod_{j=1}^k\frac{\mathrm{Tr}(\Lambda_jUAU^*)^{\ell_j}}{\ell_1!\dots\ell_k!} dU\\
        =\sum_{\ell_1,\dots ,\ell_k =0}^{\infty} \sum_{\tau_1\vdash\ell_1,\dots ,\tau_k\vdash\ell_k} \int_{\mathbf{U}(m)}\prod_{j=1}^k\frac{\tilde{c}_{\tau_j}}{\ell_j!}s_{\tau_j}(\Lambda_jUAU^*) dU\\
        =\sum_{\ell_1,\dots ,\ell_k =0}^{\infty}\sum_{\phi\in (\tau_1,\dots ,\tau_k)}\frac{\tilde{c}_{\tau_1}\dots \tilde{c}_{\tau_k}}{\ell_1!\dots\ell_k!}\frac{1}{s_{\phi}(I_m)}s_{\phi}^{\tau_1,\dots ,\tau_k} (\Lambda_1,\dots ,\Lambda_k) s_{\phi}^{\tau_1,\dots ,\tau_k}(A,\dots ,A)
    \end{multline*}
    for some constants $\tilde{c}_{\tau_j}$.
    Note that $A\mapsto s_{\phi'}^{\tau_1,\dots ,\tau_k}(A,\dots ,A)$ is invariant under conjugation by unitary matrices and an element of $ V_{\phi}\otimes\overline{V}_{\phi}$ by Schur's lemma, since the map $p:\mathbf{P}_{r_1,\dots ,r_k}(\mathrm{Her}(m)^k)\rightarrow\mathbf{P}_{r_1+\dots +r_k}(\mathrm{Her}(m))$ given by $p(Q)(\Lambda) :=Q(\Lambda ,\dots ,\Lambda )$ is an $Gl_m(\mathbb{C})$-intertwiner.
    It is therefore equal to a non-zero scalar multiple of $s_{\phi}$ for $\phi'\equiv\phi$. 
    One now obtains the statement of the lemma by comparing the coefficients of $s_{\phi}$.
\end{proof}

\begin{example}\label{example:multivariate-Shur-low-degree}
    We use the techniques described in \cite{Davis1981} to obtain the multivariate Schur polynomials in two Hermitian matrix variables of total degree $(1,1)$ for $m=2$.
    By the Littlewood--Richardson rule the only spaces appearing in the decomposition of $V_{(1)}\otimes V_{(1)}$ are $V_{(2)}^{(1),(1)}$ and $V_{(1,1)}^{(1),(1)}$ both with multiplicity one.
    Furthermore,
    \begin{equation*}
        s_{(1,1)}(\Lambda )=\det (\Lambda )=\tfrac{1}{2}(\mathrm{Tr}(\Lambda)^2 -\mathrm{Tr}(\Lambda^2 ))
    \end{equation*}
    and therefore
    \begin{equation*}
        s_{(1,1)}(\Lambda_1 +\Lambda_2 )=s_{(1,1)}(\Lambda_1) +s_{(1,1)}(\Lambda_2) +\mathrm{Tr}(\Lambda_1)\mathrm{Tr}(\Lambda_2) -\mathrm{Tr}(\Lambda_1\Lambda_2 ).
    \end{equation*}
    Therefore, up to scalar
    \begin{equation*}
        s_{(1,1)}^{(1),(1)}(\Lambda_1 ,\Lambda_2 )=\tfrac{1}{2}(\mathrm{Tr}(\Lambda_1)\mathrm{Tr}(\Lambda_2) -\mathrm{Tr}(\Lambda_1\Lambda_2))
    \end{equation*}
    and similarly
    \begin{equation*}
        s_{(2)}^{(1),(1)}(\Lambda_1 ,\Lambda_2 )=\tfrac{1}{2}(\mathrm{Tr}(\Lambda_1)\mathrm{Tr}(\Lambda_2) +\mathrm{Tr}(\Lambda_1\Lambda_2)) ,
    \end{equation*}
    using $s_{(2)}(\Lambda) =\tfrac{1}{2}(\mathrm{Tr}(\Lambda)^2 +\mathrm{Tr}(\Lambda^2 ))$.
\end{example}

The next lemma allows us to define the absolute degree of a polynomial in $\mathbf{P}(\mathrm{Her}(m)^k)^{\mathbf{U}(m)}$.

\begin{lemma}\label{lemma:reducing-variables}
    Let $1\leq\ell <k$ and $1\leq j_1<\dots <j_{\ell}\leq k$.
    Let $\mu$ be a measure on $\mathrm{Her}(m)^{k-\ell}$ invariant under simultaneous conjugation by unitary matrices and $P\in\mathbf{P}_{r_1,\dots ,r_k}(\mathrm{Her}(m)^k)^{\mathbf{U}(m)}$.
    The polynomial
    \begin{equation*}
        P_j(\Lambda_{j_1},\dots ,\Lambda_{j_{\ell}}):=\int P(\Lambda_1,\dots ,\Lambda_k)d\mu(\Lambda_1,\dots ,\Lambda_{j_1-1},\Lambda_{j_1+1},\dots ,\Lambda_{j_{\ell}-1}, \Lambda_{j_{\ell}+1},\dots ,\Lambda_k)
    \end{equation*}
    lies in $\mathbf{P}_{r_{j_1},\dots ,r_{j_{\ell}}}(\mathrm{Her}(m)^{\ell})^{\mathbf{U}(m)}$.
\end{lemma}
\begin{proof}
    For notational convenience we will assume that $j_v=v$ for $1\leq v\leq\ell$, which we can always achieve after a permutation of the variables.
    Let $U\in\mathbf{U}(m)$, the change of variables $\Lambda_{v}\mapsto U^*\Lambda_{v}U$ for $\ell\leq v\leq k$ gives
    \begin{equation*}
        P_{1,\dots ,\ell}(\Lambda_1,\dots ,\Lambda_{\ell}) =\int_{\mathrm{Her}(m)^{k-\ell}} P(\Lambda_1 ,\dots ,\Lambda_{\ell} ,U\Lambda_{\ell +1}U^* ,\dots ,U\Lambda_k U^*)d\mu(\Lambda_{\ell+1},\dots ,\Lambda_k)
    \end{equation*}
    for $\Lambda_1,\dots \Lambda_{\ell}\in\mathrm{Her}(m)$, see \cite[Theorem 3.5]{mathai1997jacobians} for the Jacobian.
    From this we see that $P_{1,\dots ,\ell}$ is invariant under simultaneous conjugation by unitary matrices and thus an element of $\mathbf{P}(\mathrm{Her}(m)^{\ell})^{\mathbf{U}(m)}$.
    The claim about the total degrees is trivial.
\end{proof}

Note that if $\ell =1$, the polynomial $P_{j_1}$ is a symmetric polynomial and its degree is always bounded by $(r_{j_1},0,\dots ,0)$ in the reverse lexicographical ordering.

\begin{definition}
    Let $P\in\mathbf{P}(\mathrm{Her}(m)^k)^{\mathbf{U}(m)}$, the \emph{absolute degree} of $P$ is given by the tuple of partitions $(\tau_1,\dots ,\tau_k;\phi )$, where
    \begin{equation*}
        \tau_j :=\sup_{\mu}\left\{\mathrm{abs\, deg}\int P(\Lambda_1,\dots ,\Lambda_k)d\mu(\Lambda_1,\dots ,\Lambda_{j-1},\Lambda_{j+1},\dots ,\Lambda_k)\right\} ,
    \end{equation*}
    where the supremum is taken with respect to the graded reverse lexicographical ordering and over all probability measures $\mu$ on $\mathrm{Her}(m)^{k-1}$ invariant under simultaneous conjugation by unitary matrices, and $\phi$ is the absolute degree of the symmetric polynomial $\Lambda\mapsto P(\Lambda ,\dots ,\Lambda )$.
\end{definition}

\begin{example}\label{ex:absolute-degree-multivariate-Schur-polynomials}
    Let $\mu$ be a measure on $\mathrm{Her}(m)^{k-1}$ invariant under simultaneous conjugation by unitary matrices, than
    \begin{equation*}
        Q_{\mu}(\Lambda_j) :=\int_{\mathrm{Her}(m)^{k-1}}s_{\phi}^{\tau_1,\dots ,\tau_k}(\Lambda_1,\dots ,\Lambda_k) d\mu (\Lambda_1,\dots ,\Lambda_{j-1},\Lambda_{j+1},\dots ,\Lambda_k)
    \end{equation*}
    lies in $V_{\tau_j}$, since $s_{\phi}^{\tau_1,\dots ,\tau_k}\in \mathcal{S}_{\phi}^{\tau_1,\dots ,\tau_k}\subseteq V_{\tau_1}\otimes\dots\otimes V_{\tau_k}$ by definition.
    Since $Q_{\mu}$ is invariant under conjugation by lemma \ref{lemma:reducing-variables} it has to be a multiple of $s_{\tau_j}$.
    If $Q_{\mu}=0$ for all such $\mu$, than one can show that $s_{\phi}^{\tau_1,\dots ,\tau_k}=0$ by approximating the uniform/Haar measure on the orbit of $\mathbf{U}(m)$ on $(\Lambda_1,\dots ,\Lambda_{j-1},\Lambda_{j+1},\dots ,\Lambda_k)$ under simultaneous conjugation.
    This would contradict the definition of the multivariate Schur polynomials.
    Together with the observation in the proof of lemma \ref{lemma:multinomial-expansion-Schur-function}, we see that the absolute degree of $s_{\phi}^{\tau_1,\dots ,\tau_k}$ is $(\tau_1 ,\dots ,\tau_k ;\phi)$.
\end{example}

One extension of the \emph{natural ordering} to the index set $\Theta_k$ is the partial ordering given by
\begin{equation*}
    \theta <\theta' \Leftrightarrow \begin{cases}\tau_{\ell} =\rho_{\ell}\textrm{ for all }j<\ell\leq k \textrm{ and }\tau_{j} <\rho_{j}\textrm{ for some } 1\leq j\leq k\\
    \textrm{or }\tau_{\ell} =\rho_{\ell}\textrm{ for all }1\leq\ell\leq k\textrm{ and } \phi <\phi' .
    \end{cases}
\end{equation*}
The natural ordering is just a combination of the natural ordering on the sets of partitions combined with some lexicographical ordering on the different partitions.
Note that in this partial ordering the indices $(\tau_1,\dots ,\tau_k ;\phi )$ with different, but equivalent partitions $\phi$ are incomparable.
Furthermore, if $(\tau_1,\dots ,\tau_k ;\phi ) >(\tau_1,\dots ,\tau_k ;\psi )$, than $(\tau_1,\dots ,\tau_k ;\phi ) >(\tau_1,\dots ,\tau_k ;\psi' )$ for all $\psi'\equiv\psi$.
This means that the notion of triangularity with respect to the multivariate Schur polynomials only depends on the specific choice of multivariate Schur polynomials in the leading coefficient.
It is quite arbitrary to define a total ordering on $\Theta_k$ as there is in general no reasonable way to separate the equivalent $\phi$'s.

We will now briefly describe the corresponding invariant orthogonal polynomials, see also \cite[\textsection 5.2.4]{hoshino2021pioneering}.
A system of invariant polynomials $(P_{\theta})_{\theta\in\Theta_k}$ of multiple matrix arguments are called \emph{orthogonal} with respect to a compactly supported weight $\mu$, invariant under simultaneous conjugation with unitary matrices, if $P_{\theta}(\Lambda_1,\dots ,\Lambda_k)$ has total degree $|\tau_j|$ in the coefficients of $\Lambda_j$ for $1\leq j\leq k$ with $\theta =(\tau_1 ,\dots ,\tau_k ;\phi )$ and furthermore
\begin{equation*}
    \int P_{\theta}(X_1,\dots ,X_k)P_{\theta'}(X_1,\dots ,X_k)\mu (dX_1,\dots ,dX_k)
    =\delta_{\theta\theta'}
    h_{\theta}
\end{equation*}
for $\theta ,\theta'\in\Theta_k$.
By a dimension argument, a system of invariant polynomials $(P_{\theta})_{\theta\in\Theta_k}$ forms a basis of $\mathbf{P}(\mathrm{Her}(m)^k)^{\mathbf{U}(m)}$. 

To deal with the non-uniqueness we strengthen the notion of triangularity for operators in such a way that it does not depend on the choice of multivariate Schur polynomials.
More precisely, that the leading coefficients will only depend on the \emph{isotypic components}.

\begin{definition}
    Let $(b_{\tau_1,\dots ,\tau_k ;\phi })_{(\tau_1,\dots ,\tau_k ;\phi )\in\Theta_k}$ be a graded basis of $\mathbf{P}(\mathrm{Her}(m)^k)^{\mathbf{U}(m)}$.
    A linear operator $\mathcal{L} :\mathbf{P}(\mathrm{Her}(m)^k)^{\mathbf{U}(m)}\rightarrow \mathbf{P}(\mathrm{Her}(m)^k)^{\mathbf{U}(m)}$ is called \emph{strongly triangular} with respect to $(b_{\tau_1,\dots ,\tau_k ;\phi })_{(\tau_1,\dots ,\tau_k ;\phi )\in\Theta_k}$ if
    \begin{equation*}
        \mathcal{L}b_{\tau_1,\dots ,\tau_k ;\phi}=\sum_{(\rho_1,\dots ,\rho_k ;\psi )\leq (\tau_1,\dots ,\tau_k ;\phi)} a_{\tau_1,\dots ,\tau_k ;\phi}^{\rho_1,\dots ,\rho_k ;\psi} b_{\rho_1,\dots ,\rho_k ;\psi}
    \end{equation*}
    and $a_{\tau_1,\dots ,\tau_k ;\phi}^{\tau_1,\dots ,\tau_k ;\phi} = a_{\tau_1,\dots ,\tau_k ;\phi'}^{\tau_1,\dots ,\tau_k ;\phi'}$ whenever $\phi\equiv\phi'$.
\end{definition}

\subsection{Koornwinder's method}
Let $P$ be an invariant polynomial of one Hermitian matrix variable, $B$ a Hermitian matrix and $A$ a positive Hermitian matrix.
Note that $A^{-\frac{1}{2}}BA^{-\frac{1}{2}}$ is again a Hermitian matrix and can thus be taken as argument of $P$.
Note also that $A^{-\frac{1}{2}}BA^{-\frac{1}{2}}$ and $BA^{-1}$ have the same eigenvalues\footnote{This follows from the general fact that for any two matrices $A$ and $B$, $AB$ and $BA$ have the same eigenvalues.}, we can therefore extend the expression $P(A^{-\frac{1}{2}}BA^{-\frac{1}{2}})$ to invertible Hermitian matrices in $A$ by seeing it as the symmetric polynomial in the eigenvalues of $BA^{-1}$.
We use continuity and the fact that the invertible Hermitian matrices lie dense in the space of all Hermitian matrices to make sense of this expression for non-invertible Hermitian matrices, whenever possible.

\begin{lemma}\label{lemma:product-polynomials-Hermitian}
    Let $P$ be an invariant polynomial of one Hermitian matrix argument of absolute degree $\tau =(\tau_1,\dots ,\tau_m )$, then the function
    \begin{align*}
        Q(A,B) :=\det (A)^{\tau_1}P(A^{-\frac{1}{2}}BA^{-\frac{1}{2}}) 
        =\det (A)^{\tau_1}P(BA^{-1})
    \end{align*}
    is an element of $P_{m\tau_1 -|\tau|,|\tau|}(\mathrm{Her}(m)^2)^{\mathbf{U}(m)}$.
\end{lemma}
\begin{proof}
    We will first show that $Q$ is invariant under simultaneous conjugation by unitary matrices. 
    The relation $CA^{\frac{1}{2}}(CA^{\frac{1}{2}})^* =CAC^*$ for positive $A\in\mathrm{Her}(m)$ and $C\in\mathrm{GL}_m(\mathbb{C})$ gives $(CAC^*)^\frac{1}{2} =CA^{\frac{1}{2}}V$ for some $V\in\mathbf{U}(m)$ or written differently
    \begin{equation}\label{eq:polar-decomposition-conjugation}
        (CAC^*)^{-\frac{1}{2}}=V^*A^{-\frac{1}{2}}C^{-1}\textrm{ and }
        (CAC^*)^{-\frac{1}{2}}=(C^*)^{-1}A^{-\frac{1}{2}}V .
    \end{equation}
    These relations with $A=A$ and $C=U$ (in which case $V=U^*$) show that $Q$ is invariant under the transformation
    \begin{equation*}
        Q(A,B)\mapsto Q(UAU^*,UBU^*)
    \end{equation*}
    for $U\in\mathbf{U}(m)$.

    Since the Schur polynomials $(s_{\rho})_{\rho\in\mathcal{P}}$ form a basis of the symmetric polynomials it is enough to prove that $Q$ is a polynomial for $P=s_{\rho}$ with $\rho_1\leq_{\mathrm{rlg}} r:=\tau_1$.
    Evaluating the second argument in $UBU^*$ and integrating over $U$ with respect to the Haar measure on $\mathbf{U}(m)$ gives
    \begin{align*}
        \tilde{Q}(A,B):=&\int_{\mathbf{U}(m)}Q(A,UBU^*)dU \\
        =&\det (A)^r\int_{\mathbf{U}(m)}s_{\rho}(A^{-1}UBU^*)dU 
        =\det (A)^{r}\frac{s_{\rho}(A^{-1})s_{\rho}(B)}{s_{\rho}(I_m)}
    \end{align*}
    by the first formula in equation \eqref{eq:integral-formulas-Schur-functions}.
    The function $\tilde{Q}$ is clearly a polynomial in two Hermitian matrix arguments since the eigenvalues of $A^{-1}$ are given by the reciprocals of those of $A$ and are raised to a power $\leq\rho_1\leq r$.
    More precisely, it is a product of two invariant polynomials with one Hermitian matrix argument and thus an invariant polynomial in the coefficients of these matrices.
    
    Now we use the fact that
    \begin{align*}
        \int_{\mathbf{U}(m)}\det (U)^r\tilde{Q}(U^*A,BU^*)dU =&\int_{\mathbf{U}(m)}\det (U)^r\det (U^*A)^{r}\frac{s_{\rho}(A^{-1}U)s_{\rho}(BU^*)}{s_{\rho}(I_m)}dU \\
        =&\frac{\det (A)^r}{s_{\rho}(I_m)} \int_{\mathbf{U}(m)}s_{\rho}(A^{-1}U)s_{\rho}(BU^*) dU\\
        =&\frac{\det (A)^r}{s_{\rho}(I_m)^2} s_{\rho}(A^{-1}B) 
        =\frac{Q(A,B)}{s_{\rho} (I_m)^2}
    \end{align*}
    to see that $Q$ is also a polynomial with two Hermitian matrix arguments, where we used that the determinant is multiplicative and the second formula in equation \eqref{eq:integral-formulas-Schur-functions}.
    Since it is a polynomial in the coefficients of the matrix it also follows that it can be extended to all Hermitian matrices by a density argument.
\end{proof}

\begin{remark}\label{rmk:cancelation}
    In \cite{huh2022logarithmic}, the following interesting identity is proved
    \begin{equation*}
        \det (A)^{r}s_{\tau}(A^{-1}) =s_{(r-\tau_m,\dots ,r-\tau_1)}(A).
    \end{equation*}
    for partitions $\tau$ and $r\geq\tau_1$.
\end{remark}

\begin{example}\label{example:simplest-example}
    We will look at the polynomial of multiple matrix arguments given by
    \begin{equation*}
        Q(A,B):=\det (A)\mathrm{Tr}(A^{-\frac{1}{2}}BA^{-\frac{1}{2}})
        =\det (A)s_{(1)}(BA^{-1}) .
    \end{equation*}
    This is indeed a polynomial of multiple matrix arguments, since the cofactor matrix $C:=\det (A)(A^{-1})^T$ is by Cramer's rule equal to the minor matrix up to sign.
    Therefore,
    \begin{equation*}
        Q(A,B)=\det (A)\sum_{p,q=1}^mB_{pq}(A^{-1})_{qp} =\sum_{p,q=1}^mB_{pq}C_{pq}
        =\sum_{p,q=1}^m(-1)^{p+q}B_{pq}\det (\hat{A}_{p,q}),
    \end{equation*}
    where $\hat{A}_{p,q}$ is the matrix $A$ without the $p^{\mathrm{th}}$ row and $q^{\mathrm{th}}$ column.

    For $m=2$ we have
    \begin{equation*}
        \det(B)\mathrm{Tr}(AB^{-1}) =A_{11}B_{22} -A_{12}B_{21} -A_{21}B_{12} +A_{22}B_{11} =\mathrm{Tr}(A)\mathrm{Tr}(B) -\mathrm{Tr}(AB) ,
    \end{equation*}
    which is precisely $s^{(1),(1)}_{(1,1)}(A,B)$.
\end{example}

The technique used in example \ref{example:simplest-example} can be used to give a generalisation of lemma \ref{lemma:product-polynomials-Hermitian} to more variables.
The proof is more useful for calculations, but less abstract and therefore gives us less insight in the structure of invariant polynomials of multiple Hermitian variables.
It also gives a less optimal constant in the exponent of the determinant than lemma \ref{lemma:product-polynomials-Hermitian}.
The latter is related to the fact that the (absolute) degree of a power sum symmetric polynomial is of the form $(n,0,\dots ,0)$ and thus the worst possible.
It is likely that the constant in the exponent of the determinant in the lemma below can be improved to $\phi_1$ if $(\tau_1,\dots ,\tau_k ;\phi )$ is the absolute degree of $Q$.
This conjecture can easily be proved with an analogue of the second formula in equation \ref{eq:integral-formulas-Schur-functions} for multivariate Schur functions by extending the argument in the proof of lemma \ref{lemma:product-polynomials-Hermitian}.

\begin{lemma}\label{lemma:product-polynomials-Hermitian-generalisation}
    Let $P\in\mathbf{P}_{r_2,\dots ,r_k}(\mathrm{Her}(m)^{k-1})^{\mathbf{U}(m)}$.
    The function
    \begin{align*}
        Q(\Lambda_1,\dots ,\Lambda_k ):=\det (\Lambda_1)^{r_2+\dots +r_k}P(\Lambda_1^{-\frac{1}{2}}\Lambda_2\Lambda_1^{-\frac{1}{2}} ,\dots ,\Lambda_1^{-\frac{1}{2}}\Lambda_k\Lambda_1^{-\frac{1}{2}} )
    \end{align*}
    is an element of $\mathbf{P}_{(m-1)(r_2+\dots +r_k),r_2,\dots ,r_k}(\mathrm{Her}(m)^{k})^{\mathbf{U}(m)}$.
\end{lemma}
\begin{proof}
    Note that $Q$ is invariant under simultaneous conjugation by the same argument as in lemma \ref{lemma:product-polynomials-Hermitian}.
    By linearity it is enough to prove the claim for
    \begin{align*}
        P(\Lambda_2,\dots ,\Lambda_k):=\prod_{j=1}^q\mathrm{Tr}\left(\prod_{\ell =1}^p\Lambda_2^{a_{2\ell j}}\cdots \Lambda_{k}^{a_{k\ell j}}\right)
    \end{align*}
    with $M:=\sum_{i=2}^k\sum_{j=1}^{p}\sum_{\ell =1}^q a_{ij\ell }\leq r_2+\dots +r_k$.
    In this case
    \begin{align*}
        Q(\Lambda_1,\dots ,\Lambda_k ):=&\det (\Lambda_1)^{r_2+\dots +r_k}\prod_{j=1}^q\mathrm{Tr}\left(\prod_{\ell =1}^p(\Lambda_1^{-\frac{1}{2}}\Lambda_2\Lambda_1^{-\frac{1}{2}})^{a_{2\ell j}}\cdots (\Lambda_1^{-\frac{1}{2}}\Lambda_k\Lambda_1^{-\frac{1}{2}})^{a_{k\ell j}}\right) \\
        =&\det (\Lambda_1)^{r_2+\dots +r_k }\prod_{j=1}^q\mathrm{Tr}\left(\prod_{\ell =1}^p((\Lambda_2\Lambda_1^{-1})^{-a_{2\ell j}}\cdots (\Lambda_k\Lambda_1^{-1})^{-a_{k\ell j}})\right) \\
        =&\det (\Lambda_1)^{r_2+\dots +r_k-M}\prod_{j=1}^q\mathrm{Tr}\left(\prod_{\ell =1}^p((\Lambda_2 C^T)^{a_{2\ell j}}\cdots (\Lambda_k C^T)^{a_{k\ell j}})\right) ,
    \end{align*}
    where in the last equality we used the cofactor matrix $C:=\det (\Lambda_1 )(\Lambda_1^{-1})^T$.
    The result now follows from the fact that $C$ is equal to the minor matrix up to sign by Cramer's rule.
\end{proof}

\begin{remark}
    The proof of the above lemma also shows that we can extend $Q$ to non-positive, and even non-invertible Hermitian matrices $\Lambda_1$.
\end{remark}

\begin{remark}\label{rmk:k=2-polynomials}
Note that when $k=2$, we have
\begin{equation*}
    Q(\Lambda_1 ,F(\Lambda_1)\Lambda_2F(\Lambda_1))
    =Q(\Lambda_1 ,\Lambda_2F(\Lambda_1)^2)
\end{equation*}
for any $F:\mathbb{R}\rightarrow\mathbb{R}$ using the Borel functional calculus.
This is straightforward for polynomials of the form \eqref{eq:products-of-traces} since $F(\Lambda_1)$ and $\Lambda_1$ commute.
\end{remark}

The following proposition is a matrix-variate version of Koornwinder's method, see \cite[\textsection 3.7.2]{KOORNWINDER1975435} and \cite[Proposition 3.1]{Griffiths2011} for the scalar variate case.

\begin{proposition}\label{prop:naive-Koornwinders-method} 
    Let $(X,Y)$ be a $\mathrm{Her}(m)^{d}\times\mathrm{Her}(m)^k$-valued random variable with distribution $W$ invariant under simultaneous conjugation by unitary matrices and let $\rho :\mathbb{R}^d\rightarrow\mathbb{R}$ be a multivariate polynomial of order at most one such that $\rho (X)>0$ almost surely.
    Assume that the random variable
    \begin{equation*}
        Z:=(\rho (X))^{-\frac{1}{2}}Y(\rho (X))^{-\frac{1}{2}}
    \end{equation*}
    is independent of $X$ and denote the marginal distributions of $X$ and $Z$ by $W_X$ and $W_Z$ respectively.
    We will also assume that $W_X$ and $W_Z$ are independent under conjugation by unitary matrices.
    A family of invariant polynomials, pairwise orthogonal with respect to $W$, is given by
    \begin{equation*}
        G_{(\theta_1 ,\theta_2 )} (x,y) :=P_{\theta_1}^{M_{\theta_2}}(x)\det(\rho (x))^{M_{\theta_2}} R_{\theta_2 }\left( (\rho (x))^{-\frac{1}{2}} y(\rho (x))^{-\frac{1}{2}}\right)
    \end{equation*}
    for $(x,y)\in\mathrm{Her}(m)^d\times\mathrm{Her}(m)^k$ such that $\rho (x)>0$, $\theta_1\in\Theta_{d}$, $\theta_2\in\Theta_k$ where $(P_{k}^{j})_{k\in\Theta_d}$ and $(R_{\ell})_{\ell\in\Theta_k}$ are systems of orthogonal invariant polynomials of Hermitian matrix arguments with respect to the weights $\det(\rho (x))^{2j} W_X$ and $W_Z$ respectively and $M_{\ell}$ is the total degree of $R_{\ell}$.
    Furthermore, if $k=1$, the above statements hold with $M_{\ell} =\nu_1$ if $\nu$ is the absolute degree of $R_{\ell}$.
\end{proposition}
\begin{proof}
    We modify the proof of \cite[Proposition 3.1]{Griffiths2011} using the coordinate transform
    \begin{equation*}
        Z\mapsto (\rho (X))^{-\frac{1}{2}}Y(\rho (X))^{-\frac{1}{2}} ,
    \end{equation*}
    which has Jacobian equal to $\det (\rho (X))^{-mk}$, see \cite[Theorem 3.5]{mathai1997jacobians}.

    More precisely, note that the assumption of conditionality implies that
    \begin{equation*}
        W(dx,dy)=W_X(dx) W_{Z}(\det (\rho (x))^{-mk} dy) .
    \end{equation*}
    Therefore, for $j,r\in\Theta_d$ and $\ell ,s\in\Theta_k$ we have
    \begin{align*}
        \int G_{(j,\ell )}(x,y) G_{(r,s)}(x,y)W(dx,dy)
        =&\int P_j^{M_{\ell}}(x)P_r^{M_{s}}(x)\det (\rho (x))^{M_{\ell} +M_{s}}W_X(dx) \\
        \int R_{\ell}(z)R_s(z) W_Z(dz) 
        =&\int P_j^{M_{\ell}}(x)P_r^{M_{\ell}}(x)\det (\rho (x))^{2M_{\ell}}W_X(dx)c_{\ell }\delta_{\ell s} \\
        =&b_jc_{\ell}\delta_{jr}\delta_{\ell s} ,
    \end{align*}
    where
    \begin{equation*}
        c_{\ell} :=\int R_{\ell}(z)R_{\ell}(z) W_Z(dz)\textrm{ and }
        b_j :=\int P_j^{M_{\ell}}(x)P_j^{M_{\ell}}(x)\det (\rho (x))^{2M_{\ell}} W_X(dx).
    \end{equation*}
    That these are invariant polynomials follows directly from lemma \ref{lemma:product-polynomials-Hermitian-generalisation} and for $k=1$ from lemma \ref{lemma:product-polynomials-Hermitian}.
\end{proof}

\begin{remark}
    Note that proposition \ref{prop:naive-Koornwinders-method} does not necessarily construct a basis of orthogonal polynomials on the space $P(\mathrm{Her}(m)^{k+d})^{\mathbf{U}(m)}$ in contrast with the scalar version of Koornwinder's method.
    Indeed, we do not have $\#\Theta_{k+d}=\#\Theta_k\cdot\#\Theta_d$ when $m>1$. 

    This problem already presents itself for the product of two measures on sets of multiple Hermitian matrices invariant under simultaneous conjugation.
    The product of invariant orthogonal polynomials are obvious sets of orthogonal polynomials, but they do not span the whole space.
    In this case this isn't all that surprising since the product measure has additional symmetries, which are reflected by these solutions.
    Indeed, both the product measure and the product of polynomials are invariant under conjugation of unitary matrices of the sets of variables separately.
    We interpret the polynomials in proposition \ref{prop:naive-Koornwinders-method} as the analogue of these product solutions for the distribution $W$.
\end{remark}

\subsection{Multivariate Hermitian Jacobi polynomials}
In this subsection we will investigate orthogonal polynomials with respect to a product of complex matrix-variate Beta distributions as an illustrative example.
Let $\alpha =(\alpha_1,\dots ,\alpha_{k})$, $\beta =(\beta_1,\dots ,\beta_k)$ be such that $\alpha_j,\beta_j >m/2-1$ for $1\leq j\leq k$.
The \emph{multivariate complex matrix-variate Beta distributions} is the measure
\begin{equation*}
    MB_{k,m}^{(\alpha,\beta)} (d\Lambda_1,\dots ,d\Lambda_k) :=B^{(\alpha_1,\beta_1)}_m(d\Lambda_1)\cdots B_m^{(\kappa_k ,a_k)}(d\Lambda_k) ,
\end{equation*}
where $B_m^{(\kappa_1,\kappa_2)}$ is the complex matrix-variate Beta distribution given in \eqref{eq:complex-matrix-variate-Beta-distribution}.
The operator $\bigotimes_{j=1}^k\mathcal{L}_{\alpha_j,\beta_j}^m$, with $\mathcal{L}_{\alpha_j,\beta_j}^m$ as in \eqref{eq:HO-operator}, is symmetric with respect to $MB_{k,m}^{(\alpha,\beta)}$.

Let $(P_{\tau}^{(\alpha_j,\beta_j ,2)})_{\tau\in\mathcal{P}}$ be the Hermitian Jacobi polynomials of index $(\alpha_j ,\beta_j)$.
The \emph{product polynomials}
\begin{equation*}
    \left(\bigotimes_{j=1}^kP_{\tau_j}^{(\alpha_j,\beta_j ,2)} \right)_{\tau_1,\dots ,\tau_k\in\mathcal{P}}
\end{equation*}
are obviously pairwise orthogonal and invariant under simultaneous conjugation by unitary matrices.
Furthermore, they are clearly eigenfunctions of the operator $\bigotimes_{j=1}^k\mathcal{L}_{\alpha_j,\beta_j}^m$.
Note however that they do not form a basis of the space $\mathbf{P}(\mathrm{Her}(m)^k)^{\mathbf{U}(m)}$ when $m>1$ and $k>1$.
We will now show that they are still orthogonal to any polynomial in $\mathbf{P}(\mathrm{Her}(m)^k)^{\mathbf{U}(m)}$ of smaller absolute degree.

Let $\tau_1,\dots ,\tau_k\in\mathcal{P}$ be arbitrary.
Note that by highest-weight theory \cite[Chapter 3]{Goodman2009-hn} there exists a unique partition $\phi\in (\tau_1,\dots ,\tau_k)$ such that $\phi\geq\psi$ for all $\psi\in (\tau_1,\dots ,\tau_k)$, which we will call the maximal $\phi\in (\tau_1,\dots ,\tau_k)$.
Explicitly, it is given by $\phi =\tau_1+\dots +\tau_k$.
In particular, the space $\mathcal{V}_{\phi}^{\tau_,\dots ,\tau_k}$ defined in \eqref{eq:equivalent-partions} is one dimensional for this maximal partition $\phi\in (\tau_1,\dots ,\tau_k)$.
Furthermore, the product polynomials are triangular with respect to the multivariate Schur polynomials with degree $(\tau_1,\dots ,\tau_k;\phi )$.
Indeed, using \eqref{eq:product-Schur-polynomials} and the fact that the Hermitian Jacobi polynomials are triangular with respect to the Schur polynomials, we see
\begin{equation}\label{eq:triangularity-product-Hermitian-Jacobi}
\begin{split}
    \bigotimes_{j=1}^kP^{(\alpha_j,\beta_j ,2)}_{\tau_j}
    =&\sum_{\rho_1\leq\tau_1,\dots ,\rho_k\leq\tau_k}a_{\rho_1,\dots ,\rho_k} s_{\rho_1}\otimes\dots\otimes s_{\rho_k} \\
    =&\sum_{\rho_1\leq\tau_1,\dots ,\rho_k\leq\tau_k}\sum_{\psi\in (\rho_1,\dots ,\rho_k)}a_{\rho_1,\dots ,\rho_k}\frac{s_{\psi}^{\rho_1,\dots ,\rho_k}(I_m,\dots ,I_m)}{s_{\psi}(I_m)}s_{\psi}^{\rho_1,\dots ,\rho_k}
\end{split}
\end{equation}
for some constants $a_{\rho_1,\dots ,\rho_k}$.

Now we want to show that these product polynomials are orthogonal to all the multivariate Schur polynomials of lower absolute degree.

\begin{proposition}\label{prop:orthogonality-product-polynomials}
    Let $\tau_1,\dots ,\tau_k\in\mathcal{P}$, then
    \begin{equation*}
        \int_{\mathrm{Her}(m)^k} P_{\tau_1}^{(\alpha_1 ,\beta_1 ,2)}(\Lambda_1)\cdots P_{\tau_k}^{(\alpha_k ,\beta_k ,2)}(\Lambda_k)s^{\rho_1,\dots ,\rho_k}_{\psi}(\Lambda_1 ,\dots ,\Lambda_k)MB_{k,m}^{(\alpha,\beta )} (d\Lambda_1,\dots ,d\Lambda_k) =0
    \end{equation*}
    for all $\psi\in (\rho_1,\dots ,\rho_k)$, whenever $\rho_j <\tau_j$ for some $1\leq j\leq k$.
\end{proposition}
\begin{proof}
    Suppose $\rho_j <\tau_j$ for some $1\leq j\leq k$.
    By example \ref{ex:absolute-degree-multivariate-Schur-polynomials}
    \begin{equation*}
        \int_{\mathrm{Her}(m)^{k-1}} s^{\rho_1,\dots ,\rho_k}_{\psi}(\Lambda_1,\dots ,\Lambda_k)\prod_{\ell=1;\ell\neq j}^kP_{\tau_{\ell}}^{(\alpha_{\ell} ,\beta_{\ell} ,2)}(\Lambda_{\ell})B_{m}^{(\alpha_{\ell},\beta_{\ell} )} (d\Lambda_{\ell}) =as_{\rho_j}(\Lambda_j )
    \end{equation*}
    for some $a\in\mathbb{R}$.
    The proposition follows from the fact that $P_{\tau_{j}}^{(\alpha_{j} ,\beta_{j} ,2)}$ is orthogonal to Schur polynomials of degree strictly less than $\tau_j$ with respect to $B_{m}^{(\alpha_{j},\beta_{j} )} (d\Lambda_{j})$.
\end{proof}

Using results in \cite{Bedoya2007}, one can show that
\begin{equation*}
    \int_{\mathbf{U}(m)}P^{(\alpha_1,\beta_1,2)}_{(1)}(\Lambda_1)P^{(\alpha_1,\beta_1,2)}_{(1)}(\Lambda_2)s_{(1,1)}^{(1),(1)}(\Lambda_1 ,\Lambda_2)MB_{2,2}^{(\alpha ,\beta)}(d\Lambda_1 ,d\Lambda_2)\neq 0
\end{equation*}
so we cannot extend the above theorem to the index $\phi$.

\begin{theorem}\label{thm:triangularity-multivariate-Hermitian-operator}
    The operator $\bigotimes_{j=1}^k\mathcal{L}^m_{\alpha_j ,\beta_j}$ is strongly triangular with respect to the multivariate Schur polynomials with leading coefficients
    \begin{equation*}
        a^{\tau_1,\dots ,\tau_k;\phi}_{\tau_1,\dots ,\tau_k;\phi} =\prod_{j=1}^k\sum_{\ell =1}^m\tau_{j,\ell}\left(\tau_{j,\ell} +\alpha_j+\beta_j +1+m-2\ell\right) .
    \end{equation*}
\end{theorem}
\begin{proof}
    By proposition \ref{prop:identity-on-irreducible-components}, the part of the operator $\mathcal{L}_{\alpha_j ,\beta_j}$ which preserves the total degree in the coefficients of the $j^{\textrm{th}}$-matrix variable acts as the identity times $r_{\tau_j}^{(\alpha_j,\beta_j)}$ on $V_{\tau_j}\otimes\overline{V}_{\tau_j}$.
    The part of $\bigotimes_{j=1}^k\mathcal{L}^m_{\alpha_j ,\beta_j}$ which preserves the total degree in each of the variables acts therefore as the identity times
    \begin{equation*}
        \prod_{j=1}^k\sum_{\ell =1}^m\tau_{j,\ell}\left(\tau_{j,\ell} +\alpha_j+\beta_j +1+m-2\ell\right)
    \end{equation*}
    on $V_{\tau_1}\otimes\overline{V}_{\tau_1}\otimes\dots\otimes V_{\tau_k}\otimes\overline{V}_{\tau_k}$.
    Since this does not depend on $\phi$ we are done.
\end{proof}

From proposition \ref{prop:triangularity-eigenpolynomials} a basis of eigenfunctions of $\bigotimes_{j=1}^k\mathcal{L}^m_{\alpha_j ,\beta_j}$ can be obtained by applying the Gram--Schmidt orthogonalisation process to the multivariate Schur polynomials.

\begin{definition}
    The \emph{multivariate Hermitian Jacobi polynomials} $(\tilde{P}^{(\alpha ,\beta)}_{\tau_1,\dots ,\tau_k ;\phi})_{(\tau_1,\dots ,\tau_k;\phi )\in\Theta_k}$ are defined by
    \begin{equation*}
        \tilde{P}^{(\alpha ,\beta)}_{\tau_1,\dots ,\tau_k ;\phi} =\sum_{(\rho_1,\dots ,\rho_k ;\psi)\leq (\tau_1,\dots ,\tau_k;\phi)} c^{\tau_1,\dots ,\tau_k;\phi}_{\rho_1,\dots ,\rho_k ;\psi} s_{\psi}^{\rho_1,\dots ,\rho_k}\textrm{ with }c^{\tau_1,\dots ,\tau_k;\phi}_{\tau_1,\dots ,\tau_k ;\phi} =1
    \end{equation*}
    and
    \begin{equation*}
        \int \tilde{P}^{(\alpha ,\beta)}_{\tau_1,\dots ,\tau_k ;\phi}s_{\psi}^{\rho_1,\dots ,\rho_k}dMB_{k,m}^{(\alpha,\beta)} =0\textrm{ for all }(\rho_1,\dots ,\rho_k ;\psi)< (\tau_1,\dots ,\tau_k;\phi) .
    \end{equation*}
\end{definition}

\begin{remark}
    From the definition it is clear that $\tilde{P}^{(\alpha ,\beta)}_{\tau_1,\dots ,\tau_k ;\phi}$ and $\tilde{P}^{(\alpha ,\beta)}_{\rho_1,\dots ,\rho_k ;\psi}$ are orthogonal with respect to $MB^{(\alpha ,\beta)}_{k,m}$ if $(\rho_1,\dots ,\rho_k ;\psi) <(\tau_1,\dots ,\tau_k ;\phi)$.
    For non-comparable indices this is not so clear.
    Note that the orthogonality of $\tilde{P}^{(\alpha ,\beta)}_{\tau_1,\dots ,\tau_k ;\phi}$ with $\tilde{P}^{(\alpha ,\beta)}_{\tau_1,\dots ,\tau_k ;\phi}$ for $\phi\equiv\phi'$ depends on the choice of $s_{\phi}^{\tau_1,\dots ,\tau_k}$ and $s_{\phi'}^{\tau_1,\dots ,\tau_k}$, so we cannot expect this.
\end{remark}

\begin{conjecture}
    The polynomials $\tilde{P}^{(\alpha ,\beta)}_{\tau_1,\dots ,\tau_k ;\phi}$ and $\tilde{P}^{(\alpha ,\beta)}_{\rho_1,\dots ,\rho_k ;\psi}$ are pairwise orthogonal with respect to $MB_{k,m}^{(\alpha ,\beta)}$ for non-equivalent indices.
\end{conjecture}

While there is no uniqueness if the space $\mathcal{S}_{\psi}^{\tau_1,\dots ,\tau_k}$ is not one-dimensional, everything is uniquely defined and does not depend on choices for $\phi\in (\tau_1,\dots ,\tau_k)$ such that $\mathcal{S}_{\phi}^{\tau_1,\dots ,\tau_k}$ is one dimensional.
It is natural to ask if one is able to obtain explicit expressions in these cases.
Particularly interesting are the maximal element $\phi\in (\tau_1,\dots ,\tau_k)$ and the PRV component \cite{Parthasarathy1967, Dimitrov2009GeometricRO}.

\section{Hermitian Jacobi polynomials on the simplex}\label{chap:Jacobi-polynomials}
In this section we study the Jacobi operator in the simplex of Hermitian matrices which was independently introduced in \cite[\textsection 4.2.]{Songzi-Li2019} and \cite[Definition 2.3.]{kuijper2025}.
We show how this operator is related to the radial part of the Laplace--Beltrami operator on partial flag manifolds in suitable coordinates.
Furthermore, we will show that its invariant measure is a complex matrix-variate Dirichlet distribution. 

In \cite[\textsection 4.2.]{Songzi-Li2019}, it is shown that simplex of Hermitian matrices together with this operator is a \emph{polynomial model}.
For $1\times 1$-matrices the Hermitian Jacobi operator on the simplex simplifies to the Jacobi operator on the simplex, we will recall some of its properties in the first subsection.
The main goal of this section is to find a similar extension for the Hermitian Jacobi operator on the simplex, which can also be seen as a matrix extension of the Jacobi operator on the simplex.

\subsection{Jacobi polynomials on the simplex}\label{sec:scalar-Jacobi-simplex}
In this subsection we briefly recall the Jacobi polynomials on the simplex of some multi-index $\kappa =(\kappa_1 ,\dots ,\kappa_{k+1})$ with $\kappa_j >-\frac{1}{2}$.
For further details we refer to \cite[\textsection 3]{Griffiths2011} and \cite[\textsection 2.3.3]{Dunkl2014}. 
The simplex is defined by
\begin{align*}
    \Sigma_{k} :=\left\{(\lambda_1,\dots ,\lambda_k)\in\mathbb{R}^{k}\biggm|\sum_{j=1}^{k}\lambda_{j}\leq 1,\lambda_j\geq 0\textrm{ for } 1\leq j\leq k\right\}
\end{align*}
and the \emph{Jacobi operator on the simplex} of index $\kappa$ by
\begin{align*}
    \mathcal{G}_{\kappa} :=\sum_{j=1}^{k}\lambda_j (1-\lambda_j )\frac{\partial^2}{\partial\lambda_j^2} 
    +\sum_{j=1}^{k}\left(\left( \kappa_j +\frac{1}{2}\right) -\left( |\kappa | +\frac{k+1}{2}\right)\lambda_j\right)\frac{\partial}{\partial\lambda_j} -\sum_{1\leq j\neq\ell\leq k}\lambda_j\lambda_{\ell}\frac{\partial^2}{\partial\lambda_j\partial\lambda_{\ell}} ,
\end{align*}
where $|\kappa |:=\kappa_1+\dots +\kappa_{k+1}$.
Note that this operator is triangular with respect to the monomial polynomials with the (reverse) lexicographical ordering.
Since this is a total order, all polynomials are triangular, and the eigenpolynomials of $\mathcal{G}_{\kappa}$ are therefore any set of orthogonal polynomials with respect to the symmetric measure of $\mathcal{G}_{\kappa}$ by corollary \ref{cor:eigenfunctions-orthogonal-polynomials}.

The operator $\mathcal{G}_{\kappa}$ is symmetric with respect to the Dirichlet measure on $\Sigma_{k}$
\begin{equation*}
    \frac{\Gamma (|\kappa |+n/2)}{\prod_{j=1}^n\Gamma (\kappa_j +1/2)}\lambda_1^{\kappa_1-1/2}\cdots\lambda_{k}^{\kappa_{k}-1/2}\left(1-\lambda_1-\dots-\lambda_{k}\right)^{\kappa_{k+1}-1/2}d\lambda_1\cdots d\lambda_k ,
\end{equation*}
has a discrete spectrum given by $-j(j+|\kappa |+(k+1)/2-1)$ for $j\in\mathbb{N}_0$ and its eigenvalues are given by the Jacobi polynomials on the simplex.
An explicit expression for such polynomials is given by
\begin{align*}
    P_{\tau}^{(\kappa)} (\lambda_1,\dots ,\lambda_{k}) :=\frac{1}{\sqrt{C_{\tau} (\kappa)}}\prod_{j=1}^{k}\left( 1-\sum_{u =1}^{j-1}\lambda_u\right)^{\tau_j} P_{\tau_j}^{\left(\kappa_j ,\sum_{\ell =j+1}^{k+1}\kappa_{\ell} +2\sum_{\ell=j+1}^k\tau_{\ell}\right) }\left(\frac{\lambda_j}{1-\sum_{u=1}^{j-1}\lambda_u}\right)
\end{align*}
for some $\tau\in\mathbb{N}^{k}$.
This expression can be proved using the scalar version of Koornwinder's method \cite[Proposition 3.2.]{Griffiths2011}.
The corresponding eigenvalues can easily be obtained by letting the operator $\mathcal{G}_{\kappa}$ act on the highest monomial term of $P_{\tau}^{(\kappa)}$ with respect to the (reverse) lexicographical ordering, since the operator $\mathcal{G}_{\kappa}$ is triangular with respect to this ordered basis.
The eigenvalue of $P_{\tau}^{(\kappa)}$ is given by
\begin{equation*}
    -|\tau |\left(|\tau| +|\kappa|+\frac{k-1}{2}\right) .
\end{equation*}

\subsection{Squared radial coordinates on partial flag manifolds}
Let $k,m,n_1,\dots ,n_{k+1}\in\mathbb{N}$ be positive integers such that $m\leq\min (n_1,\dots ,n_{k+1})$ and define $n:=n_1+\dots +n_{k+1}$, $d_j:=n_{1}+\dots +n_{j}$ for $1\leq j\leq k+1$.
The \emph{partial flag manifold} of signature $(d_1,\dots ,d_{k})$ can be identified with the homogeneous space
\begin{equation*}
    F_{d_1,\dots ,d_k}(\mathbb{C}^n) =\frac{\mathbf{U}(n)}{\mathbf{U}(n_1)\times\dots\times\mathbf{U}(n_{k+1})},
\end{equation*}
here we identify $\mathbf{U}(n_1)\times\dots\times\mathbf{U}(n_{k+1})$ with the set of diagonal block matrices in $\mathbf{U}(n)$.

Note that the canonical projection $\pi_F:\mathbf{U}(n)\rightarrow F_{d_1,\dots ,d_k}(\mathbb{C}^n)$ is a Riemannian submersion with totally geodesic fibres isometric to $\mathbf{U}(n_1)\times\dots\times\dots\mathbf{U}(n_{k+1})$ \cite[Theorem 9.80]{Besse2007-en}.
In particular the Laplace--Beltrami operator $\Delta_{F_{d_1,\dots ,d_k}(\mathbb{C}^n)}$ on $F_{d_1,\dots ,d_k}(\mathbb{C}^n)$ is the projection of the Laplace--Beltrami operator $\Delta_{\mathbf{U}(n)}$ on $\mathbf{U}(n)$ with the bi-invariant metric induced by the Killing form, i.e. the unique operator $\Delta_{F_{d_1,\dots ,d_k}(\mathbb{C}^n)}$ satisfying
\begin{equation}\label{eq:Laplce-Beltrami-on-partial-flag}
    \Delta_{\mathbf{U}(n)}(f\circ\pi_F )=(\Delta_{F_{d_1,\dots ,d_k}}f)\circ\pi_F\textrm{ for }f\in C^{\infty}(F_{d_1,\dots ,d_k}(\mathbb{C}^n)) ,
\end{equation}
see \cite[Remark 3.1.11.]{Baudoin2024-is}.

We will look at the double coset space
\begin{equation}\label{eq:double-coset-space-partial-flag}
   \mathbf{U}(n-m)\times\mathbf{U} (m)\backslash \mathbf{U}(n)/\mathbf{U}(n_1)\times\dots\times\mathbf{U}(n_{k+1}),
\end{equation}
where we see $\mathbf{U}(n-m)\times\mathbf{U} (m)$ as a subgroup of $\mathbf{U}(n)$ with the help of the diagonal embedding $(U,V)\mapsto \begin{pmatrix} U & 0\\ 0 & V\end{pmatrix}$.
Note that if $k=1$, i.e. in the Grassmannian case, this reduces to
\begin{equation*}
    \mathbf{U}(n-m)\times\mathbf{U}(m)\backslash\mathbf{U}(n)/\mathbf{U}(n-n_2)\times\mathbf{U}(n_2) ,
\end{equation*}
which fits into the framework of the generalised Cartan decomposition of B. Hoogenboom \cite{Hoogenboom1984-ns}, \cite[Part III]{HECKMAN1995} and \cite[\textsection 3 Example 3]{MATSUKI199749}.
If additionally, $m=n_2$ the left and right action are by the same group and it reduces to the Cartan decomposition for symmetric spaces.

\begin{definition}
    The \emph{spherical functions} associated to $F_{d_1,\dots ,d_k}(\mathbb{C}^n)$ are the complex valued functions on $\mathbf{U}(n)$ invariant under right multiplication by $\mathbf{U}(n_1)\times\dots\times\mathbf{U}(n_{k+1})$ and the \emph{zonal spherical functions} associated to $F_{d_1,\dots ,d_{k}}(\mathbb{C}^n)$ are the spherical functions which are also invariant under right multiplication with $\mathbf{U}(n-m)\times\mathbf{U}(m)$.
\end{definition}

Spherical functions associated to $F_{d_1,\dots ,d_{k}}(\mathbb{C}^n)$ can be identified with functions on $F_{d_1,\dots ,d_{k}}(\mathbb{C}^n)$, while zonal spherical function associated to $F_{d_1,\dots ,d_{k}}(\mathbb{C}^n)$ can be identified with functions on the double coset space \eqref{eq:double-coset-space-partial-flag}.

We will now identify the zonal spherical functions associated to $F_{d_1,\dots ,d_{k}}(\mathbb{C}^n)$ with functions on the simplex of Hermitian matrices invariant under simultaneous conjugation with unitary matrices.

To do this we will parametrise (a dense subset of) the double coset space \eqref{eq:double-coset-space-partial-flag} using \emph{squared radial coordinates} on the domain
\begin{equation*}
    \mathcal{D}_m:=\left\{\begin{pmatrix}
        W_1 & \dots & W_{k+1} \\ Z_1 &\dots & Z_{k+1}
    \end{pmatrix}\in\mathbf{U}(n)\bigm| Z_{j}\in\mathbb{C}^{m\times n_j}, \det (Z_jZ_j^*)\neq 0, 1\leq j\leq k+1\right\} ,
\end{equation*}
by the smooth map
\begin{equation}\label{eq:submersion-radial-coordinates}
    p_m\begin{pmatrix} W_1 & \dots & W_{k+1} \\ Z_1&\dots &Z_{k+1}\end{pmatrix} :=(Z_1Z_1^*,\dots ,Z_{k+1}Z_{k+1}^*) .
\end{equation}
Note that $\mathcal{D}_m$ is an open and dense subset of $\mathbf{U}(n)$ by the assumption $m\leq\min (n_1,\dots ,n_{k+1})$.

Recall the \emph{(adjoint of the) Stiefel manifold}
\begin{equation*}
    V_{n,m}^*(\mathbb{C}) :=\{ M\in\mathbb{C}^{m\times n}\mid MM^*=I_m\}\cong\mathbf{U}(n-m)\backslash\mathbf{U} (n).
\end{equation*}
We will now show that the simultaneous invariants of the Hermitian matrices $p(Z_1,\dots ,Z_{k+1})$ give coordinates on the double coset space \eqref{eq:double-coset-space-partial-flag}.
Note that we can identify $\mathbf{U}(n-m)\backslash \mathbf{U}(n)$ with the adjoint of the Stiefel manifold by looking at the last $m$ rows of a unitary matrix, which we will denote by $(Z_1 ,\dots ,Z_{k+1})\in V^*_{n,m}(\mathbb{C})$ with $Z_j\in\mathbb{C}^{m\times n_j}$ for $1\leq j\leq k+1$.
Multiplying from the right by $U_1\oplus\dots\oplus U_{k+1}\in\mathbf{U}(n_1)\times\dots\times\mathbf{U}(n_{k+1})$ is equivalent to multiplying $Z_j$ by a $n_j\times n_j$ unitary matrix $U_j$ from the right for $1\leq j\leq k+1$, which allows us to recover $Z_j$ from $\Lambda_j :=Z_jZ_j^*$.
This follows from the polar decomposition of rectangular matrices \cite[Theorem 1.1.]{Higham1986} and the fact that a semi-unitary matrix can always be extended to a unitary one by the Gram--Schmidt orthogonalisation process.
Note that to make sure that $Z_jZ_j^*=\tilde{Z}_j\tilde{Z}_j^*$ if and only if $Z_j =\tilde{Z}_jU_j$ for some $U_j\in\mathbf{U}(n_j)$, we need that $m\leq n_j$, since in this case $\mathbf{U}(n_j)$ works transitively on the (adjoint of the) Stiefel manifold $V_{n_j,m}^*(\mathbb{C})$ by right multiplication as can easily be seen from its quotient representation.
Since $\Lambda_{k+1}=I_{m}-\sum_{j=1}^k\Lambda_k$, we can also restrict ourselves to $(\Lambda_1,\dots ,\Lambda_{k})$.
Left multiplication by $I_{n-m}\oplus U$ for some $U\in\mathbf{U}(m)$ translates to freedom of simultaneously conjugation by unitary matrices, which is the same as looking at the simultaneous invariants.
Finally, we note that the $\Lambda_j$ are positive Hermitian matrices whose sum is bounded by $I_m$ and that we can obtain any such $k$-tuple of Hermitian matrices.

In particular, we can see zonal spherical functions (restricted to $\mathcal{D}_m$) as functions on the \emph{simplex of Hermitian matrices}
\begin{equation*}
    \Sigma^m_k :=\left\{ (\Lambda_1,\dots ,\Lambda_{k})\in\mathrm{Her}(m)^k\biggm| \sum_{j=1}^{k}\Lambda_j\leq I_m\textrm{ and }\Lambda_j\geq 0 \,\textrm{ for }1\leq j\leq k\right\}
\end{equation*}
(restricted to $(\Sigma^m_k)^{\circ}$) invariant under simultaneous conjugation by unitary matrices.
Restricted to polynomials, this translates to the isomorphism of vector spaces
\begin{equation*}
    {}^{\mathbf{U}(n-m)\times\mathbf{U}(m)}\mathbf{P}(\mathcal{D}_m)^{\mathbf{U}(n_1)\times\dots\times\mathbf{U}(n_{k+1})}\cong\mathbf{P}((\Sigma_k^m)^{\circ})^{\mathbf{U}(m)} .
\end{equation*}
Care is to be taken when trying to extend this isomorphism from $\mathcal{D}_m$ to $\mathbf{U}(n)$ as one has to constrain the behaviour of the functions on the boundary of $\Sigma_k^m$.

\begin{remark}
    The fact that $Z_jZ_j^*=\tilde{Z}_j\tilde{Z}_j^*$ if and only if $Z_j=\tilde{Z}_jU_j$ for some $U_j\in\mathbf{U}(n_j)$, for all $1\leq j\leq k+1$, translates to the fact that the right action of $\mathbf{U}(n_1)\times\dots\times\mathbf{U}(n_{k+1})$ on $\mathbf{U}(n-m)\times\mathbf{U}(m)\backslash\mathbf{U}(n)$ is free.
    This implies in particular that a dense open subset $\hat{V}$ of the double coset space \eqref{eq:double-coset-space-partial-flag} has a unique Riemannian manifold structure, making the canonical projection $\pi_{R,m}:\mathcal{D}_m\rightarrow\hat{V}$ into a Riemannian submersion, see \cite[9.12]{Besse2007-en}.
    In general, this submersion will neither be totally geodesic nor of mean curvature zero.
    In particular, the radial part of the Laplace--Beltrami operator on $F_{d_1,\dots ,d_k}$, as defined below, will not be equal to the Laplace--Beltrami operator on an open dense subset of the double coset space \eqref{eq:double-coset-space-partial-flag}.
\end{remark}

\subsection{Hermitian Jacobi operators on the simplex}
Define $\Lambda :=p_m(U)$ for $U\in\mathbf{U}(n)$, where $p_m$ is the Riemannian submersion given by \eqref{eq:submersion-radial-coordinates}.
In particular $\Lambda_j =Z_jZ_j^*$ for $1\leq j\leq k+1$.
We will first determine how the Laplace--Beltrami operator on $F_{d_1,\dots ,d_k}(\mathbb{C}^n)$ acts on functions only depending on $\Lambda =(\Lambda_1 ,\dots ,\Lambda_k )$.
The result of this theorem was already implicitly shown in \cite[Theorem 2.2.]{kuijper2025} in the special case $n_1=\dots =n_{k+1}=m$ using probabilistic arguments.
The full details for the probabilistic argument for the general case are to appear elsewhere, see also \cite[Proposition 5.1.]{Songzi-Li2019}.

\begin{theorem}
The Laplace--Beltrami operator on $F_{d_1,\dots ,d_k}(\mathbb{C}^n)$ acts on smooth zonal spherical functions, i.e. smooth functions depending only on $(\Lambda_1,\dots ,\Lambda_{k})$, as
\begin{align*}
    \Delta^{R,m}_{F_{d_1,\dots ,d_{k}}(\mathbb{C}^n)} :=&2\sum_{j=1}^{k}\sum_{\alpha ,\beta ,\gamma ,\delta =1}^{m} (I_m-\Lambda_j)_{\alpha\delta}\Lambda_{j,\gamma\beta}\frac{\partial^2}{\partial\Lambda_{j,\alpha\beta}\partial\Lambda_{j,\gamma\delta}} \\ 
    &-\sum_{1\leq j\neq\ell\leq k}\sum_{\alpha ,\beta ,\gamma ,\delta =1}^m (\Lambda_{j,\alpha\delta}\Lambda_{\ell ,\gamma\beta} +\Lambda_{j,\gamma\beta}\Lambda_{\ell ,\alpha\delta})\frac{\partial^2}{\partial\Lambda_{j,\alpha\beta}\partial\Lambda_{\ell ,\gamma\delta}}\\
    &+2\sum_{j=1}^{k}\sum_{\alpha ,\beta =1}^m\left( n_jI_m -n\Lambda_j\right)_{\alpha\beta} \frac{\partial}{\partial\Lambda_{j,\alpha\beta}} .
\end{align*}
This operator is called the \emph{radial part of the Laplace--Beltrami operator} on $F_{d_1,\dots ,d_k}(\mathbb{C}^n)$.
\end{theorem}
\begin{proof}
We will only give an outline of the proof for the general case.
Let $(U(t))_{t\geq 0}$ be a unitary Brownian motion starting in $\mathcal{D}_m$, i.e. a diffusion with generator $\frac{1}{2}\Delta_{\mathbf{U}(n)}$.
Since the complement of $\mathcal{D}_m$ is a polar set for unitary Brownian motion, $(U(t))_{t\geq 0}$ will stay in $\mathcal{D}_m$ for all time with probability one.
Using the polar decomposition for rectangular matrices one can show that the process $(\Lambda (t))_{t\geq 0}$, with $\Lambda (t):=p_m(U(t))$ for $t\geq 0$, satisfies the stochastic differential equation
\begin{equation*}
    d\Lambda_j =\Lambda_j^{\frac{1}{2}}\sum_{\ell =1;\ell\neq j}^{k+1}d(\gamma_{\ell j})^*\Lambda_{\ell}^{\frac{1}{2}}
    +\sum_{\ell =1;\ell\neq j}^{k+1}\Lambda_{\ell}^{\frac{1}{2}}d\gamma_{\ell j}\Lambda_{j}^{\frac{1}{2}}
    +2(n_j I_m-n\Lambda_j)dt\textrm{ for }1\leq j\leq k+1.
\end{equation*}
The corresponding generator is exactly $\frac{1}{2}\Delta^{R,m}_{F_{d_1,\dots ,d_{k}}(\mathbb{C}^n)}$.
One now uses that the generator $\mathcal{G}$ of $(p_m(U(t)))_{t\geq 0}$ satisfies 
\begin{equation*}
    \Delta_{\mathbf{U}(n)}(f\circ p_m) =(2\mathcal{G}f)\circ p_m\textrm{ for } f\in\mathcal{C}^{\infty} (\mathcal{D}_m)
\end{equation*}
by Itô's formula, and equation \eqref{eq:Laplce-Beltrami-on-partial-flag} to conclude.
\end{proof}

\begin{remark}
    For the \emph{quaternionic full flag manifold} one can define squared radial coordinates in a similar way.
    The radial part of the Laplace--Beltrami operator in this case can be shown to be a Jacobi operator on the simplex of index $(3/2,\dots ,3/2)$ \cite[Theorem 2.6.]{baudoin2025quaternionic}.
\end{remark}

We see that the radial part of the Laplace--Beltrami operator on $F_{d_1,\dots ,d_k}(\mathbb{C}^n)$ is a Hermitian Jacobi operator on the simplex of index $(n_1-m/2,\dots ,n_{k+1}-m/2)$.

\begin{definition}\label{def:matrix-Jacobi-process-simplex}
    Let $\kappa =(\kappa_1,\dots ,\kappa_{k+1})$ be a multi-index such that $\kappa_j >m/2 -1$ for $1\leq j\leq k+1$.
    The \emph{Jacobi operator} of index $\kappa$ on $\Sigma^m_k$ is defined by
    \begin{align*}
            \mathcal{G}_{\kappa}^m :=&\sum_{j=1}^{k}\sum_{\alpha ,\beta ,\gamma ,\delta =1}^m (I_m-\Lambda_j)_{\alpha\delta}\Lambda_{j,\gamma\beta} \frac{\partial^2}{\partial\Lambda_{j,\alpha\beta}\partial \Lambda_{j,\gamma\delta}} \\ 
            &-\frac{1}{2}\sum_{1\leq j\neq\ell\leq k}\sum_{\alpha ,\beta ,\gamma ,\delta =1}^m (\Lambda_{j,\alpha\delta}\Lambda_{\ell ,\gamma\beta} +\Lambda_{j,\gamma\beta}\Lambda_{\ell ,\alpha\delta})\frac{\partial^2}{\partial\Lambda_{j,\alpha\beta}\partial\Lambda_{\ell ,\gamma\delta}}\\
            & +\sum_{j=1}^{k}\sum_{\alpha ,\beta =1}^m\left(\left(\kappa_j +\frac{m}{2}\right) \delta_{\alpha\beta}-\left(|\kappa |+\frac{(k+1)m}{2}\right)\Lambda_{j,\alpha\beta}\right)\frac{\partial}{\partial\Lambda_{j,\alpha\beta}} ,
    \end{align*}
    where we used the notation $|\kappa |:=\kappa_1 +\dots +\kappa_{k+1}$.
    We will also use the terms \emph{Jacobi operator on the simplex of Hermitian matrices} of index $\kappa$ and \emph{Hermitian Jacobi operator on the simplex} of index $\kappa$ to refer to $\mathcal{G}_{\kappa}^m$.
\end{definition}

\begin{remark}
    The Hermitian Jacobi operators on the simplex reduce to the Jacobi operators on the simplex when $m=1$ and to the Hermitian Jacobi operators when $k=1$.
\end{remark}

Let $\kappa :=(\kappa_1 ,\dots ,\kappa_{k+1})$ be a multi-index such that $\kappa_j >m/2 -1$ for any $1\leq j\leq k+1$.
For symmetry reasons, it is sometimes useful to lift Jacobi operators in $\Sigma_n^m$ to operators in the space
\begin{align*}
    \mathcal{T}^m_{k+1} :=\left\{ (\Lambda_1,\dots ,\Lambda_{k+1})\in \mathrm{Her}(m)^{k+1}\Bigm| \sum_{j=1}^{k+1}\Lambda_j =I_m, \Lambda_j\geq 0 \textrm{ for } 1\leq j\leq k+1\right\} .
\end{align*}
It is easy to verify that the operator $\mathcal{G}_{\kappa}^m$ acts on functions depending on $(\Lambda_1,\dots ,\Lambda_k ,\Lambda_{k+1})$, where $\Lambda_{k+1} :=I_m-\sum_{u=1}^{k}\Lambda_u$, as
\begin{align*}
        \hat{\mathcal{G}}_{\kappa}^m :=&\sum_{j=1}^{k+1}\sum_{\alpha ,\beta ,\gamma ,\delta =1}^m (I_m-\Lambda_j)_{\alpha\delta} \Lambda_{j,\gamma\beta}\frac{\partial^2}{\partial\Lambda_{j,\alpha\beta}\partial\Lambda_{j,\gamma\delta}} \\ 
        &-\frac{1}{2}\sum_{1\leq j\neq\ell\leq k+1}\sum_{\alpha ,\beta ,\gamma ,\delta =1}^{m} (\Lambda_{j,\gamma\beta} \Lambda_{\ell ,\alpha\delta} +\Lambda_{j,\alpha\delta} \Lambda_{\ell ,\gamma\beta})\frac{\partial^2}{\partial\Lambda_{j,\alpha\beta} \partial\Lambda_{\ell ,\gamma\delta}}\\
        & +\sum_{j=1}^{k+1}\sum_{\alpha ,\beta =1}^m\left(\left(\kappa_j +\frac{m}{2}\right) \delta_{\alpha\beta}-\left(|\kappa |+\frac{(k+1)m}{2}\right)\Lambda_{j,\alpha\beta}\right)\frac{\partial}{\partial\Lambda_{j,\alpha\beta}}.  
\end{align*}
And the other way around, the operator $\hat{\mathcal{G}}_{\kappa}^m$ acts on functions depending only on $(\Lambda_1,\dots ,\Lambda_{k})$ as $\mathcal{G}_{\kappa}^m$.

\subsection{Invariant measure}
Now we will determine the invariant measure of the Hermitian Jacobi operator on the simplex, the form of which one can for example guess from the corresponding expressions of the Jacobi process on the simplex \cite[\textsection 2.2]{baudoin2025fullflag} and the matrix Jacobi process \cite[Remark 8.1.8.]{Baudoin2024-is}.
The next theorem was already noted in \cite{Songzi-Li2019}.

\begin{theorem}
    The Jacobi operator $\mathcal{G}_{\kappa}^m$ is symmetric with respect to the measure, whose density with respect to the Lebesgue measure on $\Sigma_{n}^m$ is given by
    \begin{align}\label{eq:matrix-Dirichlet-distribution}
        W^{(\kappa)}_m(\Lambda_1 ,\dots ,\Lambda_{k}) :=C_{\kappa}\left(\prod_{j=1}^{k}\det (\Lambda_j)^{\kappa_j -m/2}\right)\det\left( I_m-\sum_{j=1}^k\Lambda_j\right)^{\kappa_{k+1}  -m/2},
    \end{align}
    where $\kappa =(\kappa_1,\dots ,\kappa_{k+1})$ are such that $\kappa_j> m/2-1$.
    Explicitly, the operator $\mathcal{G}_{\kappa}^m$ satisfies
    \begin{align}\label{eq:def-symmetric-weight}
        \int_{\Sigma_k^m} (\mathcal{G}_{\kappa}^m f)g W^{(\kappa)}_m d\Lambda
        =\int_{\Sigma_k^m} f(\mathcal{G}_{\kappa}^mg) W^{(\kappa)}_m d\Lambda
    \end{align}
    for all smooth and compactly supported functions $f,g$.
\end{theorem}
\begin{proof}
    We will show \eqref{eq:def-symmetric-weight}, by using integration by parts on the second order terms on the left hand side, that the left hand side is symmetric in $f$ and $g$.
    Meaning that it has to be equal to the right hand side, since this is the left hand side with the roles of $f$ and $g$ interchanged.

    Because the calculation is long and tedious we will treat all terms of the generator $\mathcal{G}_{\kappa}^m$ separately.
    The term
    \begin{align*}
        I_1:=\int_{\Sigma^m_k}\sum_{j=1}^k\sum_{\alpha ,\beta ,\gamma ,\delta=1}^m (I_m-\Lambda_j)_{\alpha\delta}\Lambda_{j,\gamma\beta}\frac{\partial^2 f}{\partial\Lambda_{j,\gamma\delta}\partial\Lambda_{j,\alpha\beta} }gW_m^{(\kappa)} d\Lambda
    \end{align*}gives
    \begin{align*}
        I_1
        =&-\sum_{j=1}^k\sum_{\alpha ,\beta ,\gamma ,\delta=1}^m \int_{\Sigma^m_k}\bigg( (I_m-\Lambda_j)_{\alpha\delta}\Lambda_{j,\gamma\beta}\frac{\partial f}{\partial\Lambda_{j,\alpha\beta} }\frac{\partial g}{\partial\Lambda_{j,\gamma\delta}}\\
        &\qquad\qquad\qquad\quad+ \left(\kappa_j -\frac{m}{2}\right)(I_m-\Lambda_j)_{\alpha\delta} \Lambda_{j,\gamma\beta}\Lambda_{j,\delta\gamma}^{-1}\frac{\partial f}{\partial\Lambda_{j,\alpha\beta}} g \\
        &\qquad\qquad\qquad\quad +\left(\kappa_{k+1} -\frac{m}{2} \right)(I_m-\Lambda_j)_{\alpha\delta} \Lambda_{j,\gamma\beta}\left(I_m -\sum_{i=1}^k\Lambda_{i}\right)^{-1}_{\delta\gamma}\frac{\partial f}{\partial\Lambda_{j,\alpha\beta}} g \\
        &\qquad\qquad\qquad\quad +(-\delta_{\alpha\gamma}\Lambda_{j,\gamma\beta} +\delta_{\beta\delta} (I_m-\Lambda_j)_{\alpha\delta})\frac{\partial f}{\partial\Lambda_{j,\alpha\beta}}g\bigg) W_{m}^{(\kappa )} d\Lambda ,
    \end{align*}
    which after simplification reduces to
    \begin{align*}
        I_1=&-\sum_{j=1}^k\sum_{\alpha ,\beta ,\gamma ,\delta =1}^m\int_{\Sigma^m_k} (I_m-\Lambda_j)_{\alpha\delta}\Lambda_{j,\gamma\beta}\frac{\partial f}{\partial\Lambda_{j,\alpha\beta} }\frac{\partial g}{\partial\Lambda_{j,\gamma\delta}}W_m^{(\kappa)}  d\Lambda \\
        &-\sum_{j=1}^{k}\sum_{\alpha ,\beta =1}^m\left(\int_{\Sigma_k^m} \left(\kappa_j +\frac{m}{2}\right)(I_m-\Lambda_j)_{\alpha\beta} \frac{\partial f}{\partial\Lambda_{j,\alpha\beta}} g W^{(\kappa)}_m d\Lambda 
        +\int_{\Sigma_{k}^m} m\Lambda_{j,\alpha\beta}\frac{\partial f}{\partial\Lambda_{\alpha\beta}}gW_m^{(\kappa)}d\Lambda\right)\\
        &-\sum_{j=1}^{k}\sum_{\alpha ,\beta ,\gamma ,\delta =1}^m\int_{\Sigma_k^m}\left(\kappa_{k+1} -\frac{m}{2} \right)(I_m-\Lambda_j)_{\alpha\delta} \Lambda_{j,\gamma\beta}\left(I_m -\sum_{i=1}^k\Lambda_{i}\right)^{-1}_{\delta\gamma}\frac{\partial f}{\partial\Lambda_{j,\alpha\beta}} gW_m^{(\kappa)} d\Lambda
    \end{align*}
    and we get the same expression for the term with the roles of $\alpha ,\delta$ interchanged with $\gamma ,\beta$ respectively in the coefficients by using symmetry of second derivatives.
    For the second term, we include some terms of $I_1$
    \begin{align*}
        I_2 :=&\int_{\Sigma^m_k} \sum_{1\leq j\neq\ell\leq k}^k\sum_{\alpha ,\beta ,\gamma ,\delta=1}^m \Lambda_{j, \alpha\delta}\Lambda_{\ell ,\gamma\beta}\frac{\partial^2 f}{\partial\Lambda_{j,\gamma\delta}\partial\Lambda_{\ell ,\alpha\beta} }gW_m^{(\kappa)} d\Lambda 
        -\sum_{j=1}^k\sum_{\alpha ,\beta =1}^m\int_{\Sigma_{k}^m} m\Lambda_{j, \alpha\beta}W_m^{(\kappa)}d\Lambda\\
        &-\sum_{j=1}^{k}\sum_{\alpha ,\beta ,\gamma ,\delta =1}^m\int_{\Sigma_k^m}\left(\kappa_{k+1} -\frac{m}{2} \right)(I_m-\Lambda_j)_{\alpha\delta} \Lambda_{j,\gamma\beta}\left(I_m -\sum_{i=1}^k\Lambda_{i}\right)^{-1}_{\delta\gamma}\frac{\partial f}{\partial\Lambda_{j,\alpha\beta}} gW_m^{(\kappa)} d\Lambda
    \end{align*}
    we get
    \begin{align*}
        I_2
        =&-\sum_{\alpha,\beta,\gamma,\delta =1}^m\int_{\Sigma_k^m}\bigg(\sum_{1\leq j\neq \ell\leq k}\Lambda_{j,\gamma\beta} \Lambda_{\ell,\alpha\delta}\frac{\partial f}{\partial\Lambda_{\ell,\gamma\delta}}\frac{\partial g}{\partial\Lambda_{j,\alpha\beta}} 
        -\sum_{j=1}^k\Lambda_{j,\alpha\beta} \delta_{\beta\gamma}\frac{\partial f}{\partial\Lambda_{j,\alpha\beta}}g\\
        &\qquad\qquad\quad +\sum_{1\leq j\neq \ell\leq k}\left(\left(\kappa_j -\frac{m}{2}\right)\Lambda_{j,\gamma\beta} \Lambda_{\ell,\alpha\delta}\Lambda_{j,\beta\alpha}^{-1} +\delta_{\gamma\alpha}\Lambda_{\ell ,\alpha\delta}  \right)\frac{\partial f}{\partial\Lambda_{\ell,\gamma\delta}} g \\
        &\qquad\qquad\quad -\sum_{\ell =1}^k\left(\kappa_{k+1} -\frac{m}{2}\right)\left( I_m-\sum_{i=1}^k\Lambda_{i}\right)_{\gamma\beta} \Lambda_{\ell,\alpha\delta}\left( I_m-\sum_{i=1}^{k}\Lambda_{i}\right)_{\beta\alpha}^{-1}\frac{\partial f}{\partial\Lambda_{\ell,\gamma\delta}} g
        \bigg) W_m^{(\kappa)}d\Lambda \\
        =&\sum_{1\leq j\neq\ell\leq k}\sum_{\alpha,\beta,\gamma,\delta =1}^m \int_{\Sigma_k^m}\Lambda_{j,\gamma\beta} \Lambda_{\ell,\alpha\delta}\frac{\partial f}{\partial\Lambda_{\ell,\gamma\delta}}\frac{\partial g}{\partial\Lambda_{j,\alpha\beta}} W^{(\kappa)}_m d\Lambda  \\
        &+\sum_{j=1}^k\sum_{\gamma,\delta =1}^m\int_{\Sigma_k^m}\left( |\kappa | -\kappa_j +\frac{n}{2}\right) \Lambda_{j,\gamma\delta}\frac{\partial f}{\partial\Lambda_{j,\gamma\delta}} g W^{(\kappa)}_m d\Lambda 
    \end{align*}
    and almost\footnote{The term involving $\frac{\partial f}{\partial\Lambda_{j,\alpha\beta}} g$ will be exactly the same.} the same expression for the term with the roles of $\alpha ,\delta$ interchanged with $\gamma ,\beta$ respectively in the coefficients.
    Note that the terms involving $\frac{\partial f}{\partial\Lambda_{j,\alpha\beta}} g$ exactly cancel the first order terms in $\mathcal{G}_{\kappa}^m$.
    The conclusion now readily follows.
\end{proof}

\begin{remark}
The distribution \eqref{eq:matrix-Dirichlet-distribution} is called the \emph{(type I) complex matrix-variate Dirichlet distribution} of index $\kappa +1-m/2$ \cite{gupta2007properties}, see also \cite[Chapter 6]{Gupta2000} for the real case.
As discussed in \cite[\textsection 4.1]{Songzi-Li2019} the normalisation constant $C_{\kappa}$ can be calculated explicitly, it is equal to
\begin{align*}
    C_{\kappa} =\pi^{\frac{1}{2}km(m-1)}\frac{\prod_{j=1}^{k+1}\prod_{\ell =1}^{m}\Gamma (\kappa_j +m/2+1-\ell )}{\prod_{j=1}^{m}\Gamma (|\kappa |-(k+1)(m-2)/2 +m-j)} ,
\end{align*}
where $\Gamma$ is the Euler $\Gamma$-function.
\end{remark}

Note that if $(\Lambda_1,\dots ,\Lambda_k )$ is a random variable with density $W_m^{(\kappa )}$, then $\Lambda_j$ has a marginal density given by
\begin{align*}
    B_m^{(\kappa_j ,|\kappa |)} =C_{\kappa_j, |\kappa |}\det (\Lambda_j )^{\kappa_j -m/2}\det (I_m -\Lambda_j )^{|\kappa |-\kappa_j -m/2}
\end{align*}
for some constant $C_{\kappa_j ,|\kappa |}>0$, see \cite[Corollary 6.3.2.1.]{Gupta2000} for the real case.
This is the complex matrix-variate Beta distribution of index $(\kappa_j , |\kappa |-\kappa_j)$ defined in equation \eqref{eq:complex-matrix-variate-Beta-distribution}.

From \cite[Theorem 2.4.]{gupta2007properties}, it follows that the complex matrix-variate Dirichlet distribution is \emph{neutral to the right} or \emph{right neutral}, see \cite{Doksum1974} and the references therein for more on this concept.
We record this fact here for completeness.

\begin{proposition}\label{prop:splitting-of-measure}
    Let $\Lambda :=(\Lambda_1,\dots ,\Lambda_k)$ have a matrix-variate Dirichlet distribution of index $\kappa +1-m/2$.
    Define $b_j :=\sum_{\ell=j+1}^{k+1}\kappa_{\ell}$ for $1\leq j\leq k$ and
    \begin{equation*}
        \Xi_j :=\left( I_m-\sum_{u=1}^{j-1}\Lambda_u\right)^{-\frac{1}{2}}\Lambda_j\left( I_m-\sum_{u=1}^{j-1}\Lambda_u\right)^{-\frac{1}{2}}\textrm{ for } 1\leq j\leq k
    \end{equation*}
    as before.
    The random variables $(\Xi_j)_{j=1}^k$ are distributed according to independent complex matrix-variate Beta distributions, more precisely
    \begin{equation*}
        (\Xi_1,\dots ,\Xi_k)\sim B_m^{(\kappa_1 ,b_1)}(d\Xi_1)\cdots B_m^{(\kappa_k ,b_k)}(d\Xi_k) .
    \end{equation*}
    Furthermore, $\Xi_j$ is independent of $\Lambda_1,\dots ,\Lambda_{j-1}$ for $1\leq j\leq k$.
    
    In terms of the measure this translates to
    \begin{equation*}
    \begin{split}
        W_m^{(\kappa_1,\dots ,\kappa_k)}(d\Lambda_1 ,\dots ,d\Lambda_k)
        =&\prod_{j=1}^k B_m^{(\kappa_j ,b_j)}\left(\left(\det\left( I_m-\sum_{u=1}^{j-1}\Lambda_u\right)\right)^{-m}d\Lambda_j\right) \\
        =&\prod_{j=1}^k B_m^{(\kappa_j ,b_j)}\left(d\Xi_j\right) .
    \end{split}
    \end{equation*}
\end{proposition}

\subsection{Orthogonal polynomials}
In this section we will define a family of pairwise orthogonal polynomials invariant under simultaneous conjugation with respect to the matrix-variate Dirichlet measure.
The polynomials can be seen as analogues of the product solutions for product measures.
Furthermore, they might give some insight into explicit expressions for a more complete set of orthogonal polynomials.

\begin{theorem}\label{thm:expression-matrix-Jacobi-simplex}
    Let $k,m\geq 1$ be integers and $\kappa\in\mathbb{R}^{k+1}$ be such that $\kappa_j >m/2-1$ for $1\leq j\leq k+1$.
    For partitions $\tau_1 ,\dots ,\tau_k$ of length $m$, the functions
    \begin{multline}\label{eq:matrix-Jacobi-polynomials-simplex}
        P_{\tau_1,\dots ,\tau_k}^{(\kappa )} (\Lambda_1,\dots ,\Lambda_{k}) :=\prod_{j=1}^k\det\left( I_m-\sum_{u =1}^{j-1}\Lambda_{u}\right)^{\tau_{j,1}} \\
        P_{\tau_{j}}^{( \kappa_{j} ,a_{j} ,2)}\left(\left( I_m-\sum_{u=1}^{j-1}\Lambda_{u}\right)^{-\frac{1}{2}}\Lambda_{j}\left( I_m-\sum_{u=1}^{j-1}\Lambda_{u}\right)^{-\frac{1}{2}}\right) 
    \end{multline}
    with
    \begin{equation*}
        a_j :=2\sum_{\ell =j+1}^{k} \tau_{\ell ,1} +\sum_{\ell =j+1}^{k+1}\kappa_{\ell} ,
    \end{equation*}
    are pairwise orthogonal polynomials with respect to the complex matrix-variate Dirichlet measure $W_m^{(\kappa )}$.
    Furthermore, these polynomials are orthogonal to invariant polynomials of absolute degree in the coefficients of $\Lambda_{k}$ strictly less than $\tau_k$.
    
    The polynomials $P_{\tau_1,\dots ,\tau_k}^{(\kappa)}$ given in \eqref{eq:matrix-Jacobi-polynomials-simplex} will be called \emph{product Hermitian Jacobi polynomials on the simplex}.
\end{theorem}
\begin{proof}
    The functions $P_{\tau}^{(\kappa )}(\Lambda_1,\dots ,\Lambda_k)$ can be written recursively as follows
    \begin{multline*}
        P_{\tau}^{(\kappa )} (\Lambda_1,\dots ,\Lambda_{k}) =\det (I_m-\Lambda_1 )^{\sum_{\ell =2}^{k}\tau_{\ell ,1} }P_{\tau_1}^{(\kappa_1 ,a_1,2)}(\Lambda_1)\\
        P_{\tau_2,\dots ,\tau_k}^{(\kappa_2,\dots ,\kappa_{k+1})}\left( (I_m-\Lambda_1)^{-\frac{1}{2}}\Lambda_2 (I_m-\Lambda_1)^{-\frac{1}{2}},\dots ,(I_m-\Lambda_1)^{-\frac{1}{2}}\Lambda_{k} (I_m-\Lambda_1)^{-\frac{1}{2}}\right) .
    \end{multline*}
    To see this one can use that $P_{\tau_2,\dots ,\tau_k}^{(\kappa_2,\dots ,\kappa_{k+1})}$ is invariant under simultaneous conjugation with unitary matrices by the induction hypothesis.
    The expressions \eqref{eq:polar-decomposition-conjugation} can be used to show that $P_{\tau}^{(\kappa )}$ is invariant under simultaneous conjugation with unitary matrices and that it agrees with \eqref{eq:matrix-Jacobi-polynomials-simplex}.
    More concretely for the second part, the first relation in \eqref{eq:polar-decomposition-conjugation} shows that
    \begin{align*}
        \left( I_m -\sum_{u=2}^{j-1}\left( I_m-\Lambda_1\right)^{-\frac{1}{2}}\Lambda_{j}\left(I_m-\Lambda_1\right)^{-\frac{1}{2}}\right)^{-\frac{1}{2}}
        =V\left(I_m-\sum_{u=1}^{j-1}\Lambda_u\right)^{-\frac{1}{2}}\left( I_m-\Lambda_1\right)^{\frac{1}{2}}
    \end{align*}
    for $1\leq j\leq k$ and a similar expression when taking the complex conjugate.
    
    So we are done if we can apply proposition \ref{prop:naive-Koornwinders-method} inductively on $k$ with $\rho (\Lambda )=I_m-\Lambda$.
    Note that we need a higher power for the determinant in the statement of \ref{prop:naive-Koornwinders-method}, but this is only needed to ensure that it is in fact a polynomial, and that it is a polynomial is already clear from the expression \eqref{eq:matrix-Jacobi-polynomials-simplex} and lemma \ref{lemma:product-polynomials-Hermitian}.
    First off note that
    \begin{align}\label{eq:neutral-to-the-right}
        \Xi_j:=\left( I_m-\sum_{u=1}^{j-1}\Lambda_u\right)^{-\frac{1}{2}}\Lambda_j\left( I_m-\sum_{u=1}^{j-1}\Lambda_u\right)^{-\frac{1}{2}}
    \end{align}
    is independent of $(\Lambda_1,\dots ,\Lambda_{j-1})$ for $1\leq j\leq k$ and has a complex matrix-variate Beta density $W_m^{(\kappa_j ,\sum_{\ell =j+1}^{k+1}\kappa_{\ell})}$, see proposition \ref{prop:splitting-of-measure}.
    In particular it is independent of $\sum_{u=1}^{j-1}\Lambda_u$.
    
    Now we show the pairwise orthogonality of the terms
    \begin{align*}
        \det\Bigg( I_m -\left( I_m-\sum_{u=1}^j\Lambda_u\right)^{-\frac{1}{2}}\Lambda_j &\left( I_m-\sum_{u=1}^j\Lambda_u\right)^{-\frac{1}{2}}\Bigg)^{\sum_{\ell =j+1}^k\tau_{\ell ,1}} \\
        &P^{(\kappa_j ,a_j,2)}_{\tau_{j}}\Bigg( \left( I_m-\sum_{u=1}^j\Lambda_u\right)^{-\frac{1}{2}}\Lambda_j\left( I_m-\sum_{u=1}^j\Lambda_u\right)^{-\frac{1}{2}}\Bigg)
    \end{align*}
    for $j\in\mathbb{N}$ and $\tau_1,\dots ,\tau_k$ a sequence of partitions of length $m$.
    Let $\tau_1,\dots ,\tau_k ,\tau'_1,\dots ,\tau'_k\in\mathcal{P}$ be any two such sequences of partitions which are different from each other.
    Without loss of generality we can assume that there exists $1\leq j\leq k$ such that $\tau_{\ell} =\tau_{\ell}'$ for all $j +1\leq\ell\leq k$ and such that $\tau_{j}\neq\tau_{j}'$.
    In particular, $a_j=a_j'$.
    Multiplying the product of the corresponding polynomials by the matrix-variate Beta distribution $\Xi_j$, given in \eqref{eq:neutral-to-the-right}, gives
    \begin{equation*}
        P_{\tau_j}^{(\kappa_j, a_j,2)}(\Xi_j)P_{\tau'_j}^{(\kappa_j, a_j,2)}(\Xi_j)\det (\Xi_j)^{\kappa_j -\frac{m}{2}}\det(I_m-\Xi_j)^{\sum_{\ell =j+1}^{k+1}\kappa_{\ell} -\frac{m}{2} +2\sum_{\ell =j+1}^{k}\tau_{\ell ,q}} .
    \end{equation*}
    This is zero since $\tau_j\neq\tau_j'$ and we are integrating with respect to the weight measure $B_m^{(\kappa_j ,a_j )}$.
    This proves orthogonality, we can now apply proposition \ref{prop:naive-Koornwinders-method} inductively on $k$ to conclude that the set $(P_{\tau}^{(\kappa )})_{\tau\in\mathcal{P}^k}$ is pairwise orthogonal.

    Let $Q(\Lambda_1,\dots ,\Lambda_k)$ be a polynomial invariant under simultaneous conjugation by unitary matrices of absolute degree $(\rho_1,\dots ,\rho_k ;\phi)$ with $\rho_k<\tau_{k}$.
    Note that the function
    \begin{equation*}
        \widetilde{Q}(\Xi_{k}) :=\int Q(\Lambda_1,\dots ,\Lambda_k)\prod_{j=1}^{k-1}P^{(\kappa_j,a_j,2)}_{\tau_j}(\Xi_j) W_m^{(\kappa_j ,a_j-\sum_{\ell =j+1}^{k}\tau_{\ell ,1})}(d\Xi_j) ,
    \end{equation*}
    where we see $\Lambda_j$ as a function of $(\Xi_1,\dots ,\Xi_j)$, which is linear in $\Xi_j$, lies in $\mathbf{P}(\mathrm{Her}(m))^{\mathbf{U}(m)}$ and has absolute degree $\rho_k<\tau_{k}$.
    Now, by Fubini's theorem,
    \begin{multline*}
        \int_{\mathrm{Her}(m)^k} Q(\Lambda_1,\dots ,\Lambda_k)P_{\tau_1,\dots ,\tau_k}^{(\kappa)}(\Lambda_1,\dots ,\Lambda_k)W_{m}^{(\kappa )}(d\Lambda_1,\dots ,d\Lambda_k) \\
        =\int\widetilde{Q}(\Xi_k )P_{\tau_k}^{(\kappa_k ,a_k,2)}(\Xi_j)W_{m}^{(\kappa_k ,a_k)}(d\Xi_k) =0,
    \end{multline*}
    where we used that $a_k=\kappa_{k+1}$.
\end{proof}

\begin{remark}
    It is more difficult to see for which absolute degrees the product Hermitian Jacobi polynomials on the simplex are orthogonal to invariant polynomials in the other variables, since the change of variables \eqref{eq:neutral-to-the-right} depends non-polynomial on the other variables (for $k>2$) and the expression for $a_j$ involves the partitions $\tau_{j+1},\dots ,\tau_{k}$.
\end{remark}

It is not entirely trivial to determine the absolute degree of the product Hermitian Jacobi polynomials on the simplex and if they are triangular with respect to the multivariate Schur polynomials.

\begin{example}
    The simplest non-trivial product Hermitian Jacobi polynomial on the simplex is given by
    \begin{equation*}
    P_{(0),(1)}^{(\kappa)}(\Lambda_1 ,\Lambda_2 ) =\det (I_m-\Lambda_1 )\left(\mathrm{Tr} (\Lambda_2 (I_m-\Lambda_1)^{-1}) -\frac{m\kappa_1 +m^2/2}{\kappa_2+2\kappa_3 +m+2}\right) .
    \end{equation*}
    For $m=2$ this reduces to
    \begin{equation*}
        P_{(0),(1)}^{(\kappa)}(\Lambda_1 ,\Lambda_2 ) =\mathrm{Tr}(\Lambda_1\Lambda_2 )-\mathrm{Tr}(\Lambda_1 )\mathrm{Tr}(\Lambda_2 )
        +\mathrm{Tr}(\Lambda_1 ) -\frac{2\kappa_1 +2}{\kappa_2+2\kappa_3 +4}\det (I_2 -\Lambda_1 )
    \end{equation*}
    A straightforward, but lengthy, calculation gives
    \begin{multline*}
        \mathcal{G}_{\kappa}^2 P_{(0),(1)}^{(\kappa)}(\Lambda_1 ,\Lambda_2 ) =
        \mathrm{Tr}(\Lambda_1)\mathrm{Tr}(\Lambda_2)
        -\mathrm{Tr}(\Lambda_1\Lambda_2) +(2\kappa_1 +2)\left( 1 -\frac{2\kappa_1 +2}{\kappa_2+2\kappa_3 +4}\right) \\
        -\frac{2\kappa_1 +2}{\kappa_2+2\kappa_3 +4} (\det (I_2 -\Lambda_1 )-\mathrm{Tr}(\Lambda_1))
        +\left( |\kappa |+2\right) P_{(0),(1)}^{(\kappa)}(\Lambda_1 ,\Lambda_2 ) .
    \end{multline*}
    for $m=2$.
    The right hand side cannot be a multiple of $P_{(0),(1)}^{(\kappa)}$ if $3\kappa_1 =\kappa_3$, since in this case $P_{(0),(1)}^{(\kappa)}$ has a constant term while $\mathcal{G}_{\kappa}^2 P_{(0),(1)}^{(\kappa)}$ does not.
\end{example}

In particular, the product Hermitian Jacobi polynomials on the simplex are not always eigenpolynomials of $\mathcal{G}_{\kappa}^m$.
However, by remark \ref{rmk:cancelation}, they might still be eigenpolynomials for certain tuples of paritions.

\begin{conjecture}
    The product Hermitian Jacobi polynomials on the simplex $(P_{\tau_1,\dots ,\tau_k}^{(\kappa )})_{\tau_1,\dots ,\tau_k\in\mathcal{P}}$ are eigenfunctions of the Hermitian Jacobi operator on the simplex $\mathcal{G}_{\kappa_1,\dots ,\kappa_{k+1}}^m$ when $\tau_j =(r_j,\dots ,r_j)$ for some $r_j\in\mathbb{N}$ for $2\leq j\leq k$.
\end{conjecture}

\subsection{Eigenfunctions and spectrum}
In this section, we show that the operator $\mathcal{G}_{\kappa_1,\dots ,\kappa_{k+1}}^m$ is strongly triangular with respect to the multivariate Schur polynomials and introduce the Hermitian Jacobi polynomials on the simplex.

\begin{proposition}
    The generator $\mathcal{G}_{\kappa}^m$ is strongly triangular with respect to multivariate Schur polynomials ordered by the natural ordering with leading coefficient
    \begin{equation}\label{eq:eigenvalues-Hermitian-Jacobi-polynomials-on-the-simplex}
        -\sum_{j=1}^{m}\phi_j\left(\phi_j +|\kappa | +\frac{(k+1)m}{2} +1-2j\right).
    \end{equation}
\end{proposition}
\begin{proof}
    The operator $\mathcal{G}_{\kappa}^m$ is given by
    \begin{equation*}
        \mathcal{G}_{\kappa}^m =
        -\frac{1}{2}\sum_{1\leq j\neq\ell\leq k}(\Lambda_{j,\alpha\delta}\Lambda_{\ell,\gamma\beta} +\Lambda_{j,\gamma\beta}\Lambda_{\ell ,\alpha\delta})\frac{\partial^2}{\partial\Lambda_{j,\alpha\beta}\partial\Lambda_{\ell ,\gamma\delta}} 
        -\left( |\kappa| +\frac{(k+1)m}{2}\right)\mathcal{V}
        +\mathcal{O},
    \end{equation*}
    where $\mathcal{V}$ is the Euler vector field  and $\mathcal{O}$ are some degree lowering terms. 
    
    The first term $\mathcal{C}$ is a $Gl_m(\mathbb{C})$-intertwiner on $\mathbf{P}_{r_1,\dots ,r_k}(\mathrm{Her}(m)^k)$ with respect to the action induced by the simultaneous congruence transform.
    By Schur's lemma, $\mathcal{C}$ takes the isotypic components
    \begin{equation*}
        \mathcal{V}_{\phi}^{\tau_1,\dots ,\tau_k} =\bigoplus_{\phi'\equiv\phi} V_{\phi'}^{\tau_1 ,\dots ,\tau_k}
    \end{equation*}
    into itself.
    It therefore also takes the subspaces $\mathcal{S}_{\phi}^{\tau_1,\dots ,\tau_k}$ into themselves.
    Since this space is finite dimensional and $\mathcal{G}_{\kappa}^m$ is symmetric, we can choose an ordered basis of $\mathcal{S}_{\phi}^{\tau_1,\dots ,\tau_k}$ such that $\mathcal{C}$ is diagonal.

    In fact $-\mathcal{C}+m\mathcal{V}$ is a Casimir operator of the $Gl_m(\mathbb{C})$-action $g\mapsto (g\Lambda_1 g^*,\dots ,g\Lambda_{k+1}g^*)$.
    Indeed a set of infinitesimal generators of this action is given by
    \begin{align*}
        X_{\alpha\beta } :=&\frac{1}{2}\sum_{j=1}^k\sum_{\gamma =1}^m\left(\Lambda_{j,\alpha\gamma} \frac{\partial}{\partial\Lambda_{j,\beta\gamma}}
        +\Lambda_{j,\gamma\alpha}\frac{\partial}{\partial\Lambda_{j,\gamma\beta}}\right) \\
        Y_{\alpha\beta} :=&\frac{i}{2}\sum_{j=1}^k\sum_{\gamma =1}^m\left(\Lambda_{j,\alpha\gamma} \frac{\partial}{\partial\Lambda_{j,\beta\gamma}}
        -\Lambda_{j,\gamma\alpha}\frac{\partial}{\partial\Lambda_{j,\gamma\beta}}\right)
    \end{align*}
    for $1\leq\alpha ,\beta\leq m$ and a direct calculation gives $\sum_{\alpha ,\beta =1}^m (X_{\alpha\beta }X_{\beta\alpha} +Y_{\alpha\beta }\overline{Y}_{\beta\alpha}) =-\mathcal{C}+m\mathcal{V}$.
    The operator $\mathcal{C}$ therefore acts as the identity times $-\sum_{j=1}^m\phi_j (\phi_j +m+1-2j)+m|\phi |$ on the isotypic component $\mathcal{V}^{\tau_1,\dots ,\tau_k}_{\phi}$ by \cite[Lemma 3.3.8.]{Goodman2009-hn} and the proof of proposition \ref{prop:identity-on-irreducible-components}.
    Adding the terms allows us to conclude.
\end{proof}

\begin{definition}\label{def:Hermitian-Jacobi-polynomials-simplex}
    The \emph{Hermitian Jacobi polynomials on the simplex} $(P^{(\kappa_1 ,\dots ,\kappa_{k+1})}_{\tau_1,\dots ,\tau_{k+1};\phi})_{(\tau_1,\dots ,\tau_k ;\phi )\in\Theta_k}$ are defined by
    \begin{equation*}
        P^{(\kappa_1 ,\dots ,\kappa_{k+1})}_{\tau_1,\dots ,\tau_k ;\phi} =\sum_{(\rho_1,\dots ,\rho_k ;\psi)\leq (\tau_1,\dots ,\tau_k;\phi)} c^{\tau_1,\dots ,\tau_k;\phi}_{\rho_1,\dots ,\rho_k ;\psi} s_{\psi}^{\rho_1,\dots ,\rho_k}\textrm{ with }c^{\tau_1,\dots ,\tau_k;\phi}_{\tau_1,\dots ,\tau_k ;\phi} =1
    \end{equation*}
    and
    \begin{equation*}
        \int P^{(\kappa_1 ,\dots ,\kappa_{k+1})}_{\tau_1,\dots ,\tau_k ;\phi}s_{\psi}^{\rho_1,\dots ,\rho_k}dW_{m}^{(\kappa_1,\dots ,\kappa_{k+1})} =0\textrm{ for all }(\rho_1,\dots ,\rho_k ;\psi)< (\tau_1,\dots ,\tau_k;\phi) .
    \end{equation*}
\end{definition}

\begin{conjecture}
    The Hermitian Jacobi polynomials on the simplex are orthogonal with respect to the complex matrix-variate Dirichlet measure for non-equivalent indices.
\end{conjecture}

\begin{remark}
    Let $m\geq 1$, $k=2$ and $\kappa\in\mathbb{R}^3$ be such that $\kappa_j >m/2-1$ for $j=1,2,3$.
    The functions
    \begin{equation}\label{eq:general-Hermitian-Jacobi-polynomials-simplex}
        \det\left( I_m-\Lambda_{1}\right)^{r}
        \tilde{P}_{\tau_1,\tau_2;\phi }^{(\kappa ,a)} \left(\Lambda_1 ,(I_m-\Lambda_1)^{-\frac{1}{2}}\Lambda_2(I_m-\Lambda_1)^{-\frac{1}{2}} \right) ,
    \end{equation}
    are polynomials for $r\geq |\tau_2|$ (and probably also for $r\geq\tau_{2,1}$) see remark \ref{rmk:k=2-polynomials}, where $(\tilde{P}^{(\kappa ,a)}_{\tau_1,\tau_{2};\phi })_{(\tau_1 ,\tau_{2};\phi )\in\Theta_{k}}$ are multivariate Hermitian Jacobi polynomials of index $(\kappa ,a)$ with
    \begin{equation*}
        a_j :=2r\delta_{1j} +\sum_{\ell =j+1}^{3}\kappa_{\ell}\textrm{ for } 1\leq j\leq 2 .
    \end{equation*}
    Koornwinder's method suggest that there is some relation between the polynomials \eqref{eq:general-Hermitian-Jacobi-polynomials-simplex} and the Hermitian Jacobi polynomials on the simplex.
    Note that similar expressions do not give polynomials for $k\geq 3$.
\end{remark}

\bibliographystyle{plain}
\bibliography{reference.bib}

\end{document}